\documentclass[10pt,a4paper]{amsart}
\usepackage[utf8]{inputenc}
\usepackage{amsmath}
\usepackage{amsfonts}
\usepackage{amssymb}
\usepackage{enumerate}
\usepackage{graphicx}
\usepackage{hyperref}
\usepackage{amsthm}
\usepackage{hhline}
\usepackage{mathdots}
\usepackage{xcolor}
\usepackage{colortbl}
\usepackage{listings}
\usepackage{accents}

\usepackage{mathtools}
\usepackage{subfigure}
\numberwithin{table}{section}
\usepackage{longtable}
\usepackage{tabularx}
\newcolumntype{C}[1]{>{\centering\arraybackslash}p{#1}}
\usepackage{booktabs}
\usepackage[english]{babel}
\renewcommand{\arraystretch}{1.1}
\calclayout
\newtheorem{prop}{Proposition}[section]
\newtheorem{theorem}[prop]{Theorem}
\newtheorem{lemma}[prop]{Lemma}\newtheorem{theoremA}{Theorem}

\theoremstyle{remark}
\newtheorem{rem}[prop]{Remark}
\theoremstyle{definition}

\newtheorem{condition}{Condition}

\DeclareMathOperator{\Inn}{Inn}

\DeclareMathOperator{\Aut}{Aut}
\DeclareMathOperator{\sign}{sign}

\DeclareMathOperator{\ord}{ord}

\DeclareMathOperator{\Irr}{Irr}
\DeclareMathOperator{\W}{W}
\DeclareMathOperator{\Inf}{Inf}

\DeclareMathOperator{\Sym}{Sym}

\DeclareMathOperator{\Gal}{Gal}
\DeclareMathOperator{\ab}{ab}
\DeclareMathOperator{\PSU}{PSU}

\DeclareMathOperator{\SU}{SU}
\DeclareMathOperator{\GL}{GL}

\DeclareMathOperator{\St}{St}
\DeclareMathOperator{\SL}{SL}
\DeclareMathOperator{\PSL}{PSL}

\DeclareMathOperator{\Ind}{Ind}

\DeclareMathOperator{\Syl}{Syl}
\DeclareMathOperator{\Out}{Out}

\DeclareMathOperator{\pr}{pr}

\DeclareMathOperator{\PGL}{PGL}

\begin{document}
\title{The inductive McKay--Navarro condition for the Suzuki and Ree groups}
\author{Birte Johansson}
\address{FB Mathematik, TU Kaiserslautern, Postfach 3049, 67653 Kaiserslautern, Germany.}
\email{johansson@mathematik.uni-kl.de}
\begin{abstract}
We verify the inductive McKay--Navarro condition for the Suzuki and Ree groups for all primes as well as for $p=3$ and the groups $\PSL_3(q) $ with $q \equiv 4,7 \mod 9$, $\PSU_3(q) $ with $ q \equiv 2,5 \mod 9$, and $\mathsf{G_2(q)}$ with $q=2,4,5,7 \mod 9$. On the way, we show that there exists a Galois-equivariant Jordan decomposition for the irreducible characters of the Suzuki and Ree groups.
\end{abstract}
\maketitle

\section{Introduction}
In the representation theory of finite groups, the so-called local-global conjectures assert a relation between the representation theory of a finite group and that of some of its local subgroups. 

Let $p$ be a prime and $G$ a finite group. We say that a character of $G$ is a $p'$-character if its degree is not divisible by $p$.
The McKay conjecture claims that there is a bijection between the irreducible $p'$-characters of $G$ and the irreducible $p'$-characters of $N_G(R)$ where $R$ is a Sylow $p$-subgroup of $G$. In 2004, Navarro refined this conjecture and predicted that this bijection can be chosen equivariant under certain Galois automorphisms. In this paper, we are concerned with this conjecture, the McKay--Navarro or Galois--McKay conjecture.

The current approach to prove these conjectures uses the classification of finite simple groups: In a first step, the conjecture is reduced to a problem about simple groups, yielding a so-called inductive condition that is often stronger than the original conjecture. In a second step, this inductive condition is verified for all finite simple groups.

For the McKay--Navarro conjecture, a reduction theorem has been proven in 2019 by Navarro, Späth, and Vallejo \cite{navarro2019reduction}. The resulting inductive condition has been verified for all finite groups of Lie type in their defining characteristic  \cite{ruhstorfer2017navarro} \cite{johansson2020} and for $p=2$ and the simple groups $\mathsf{C}_n(q)$ ($n \geq 1$), $\mathsf{B}_n(q)$ ($n \geq 3$, $(n,q) \neq (3,3)$), $\mathsf{G}_2(q)$ ( $3 \nmid q$), $^3\mathsf{D}_4(q)$, $\mathsf{F}_4(q)$, $\mathsf{E}_7(q)$, and $\mathsf{E}_8(q)$ where $q$ is a power of an odd prime \cite{ruhstorfer2021inductive}. Additionally, there are some partial results for groups not of type $\mathsf{A}$ in \cite{fry2020galoisequivariant}.

The main result of this paper is the following theorem.
\begin{theoremA}\label{thmA}
The inductive McKay--Navarro condition is satisfied for
\begin{enumerate}[(i)]
\item the Suzuki and Ree groups for all primes $p$, 
\item $\PSL_3(q) $ with $q \equiv 4,7 \mod 9$ and $\PSU_3(q) $ with $ q \equiv 2,5 \mod 9$ for $p=3$,
\item $\mathsf{G_2(q)}$ with $q=2,4,5,7 \mod 9$ for $p=3$.
\end{enumerate} 
\end{theoremA}
In order to prove this, we generalize the existence of a Galois-equivariant Jordan decomposition for connected reductive groups with connected centre and Frobenius morphisms from \cite{srinivasanvinroot2019} to the Suzuki and Ree groups.
Note that the main work for the proof of the following theorem has already been done by Srinivasan and Vinroot and we only consider the case of Suzuki and Ree groups, see Proposition \ref{propjordan}.
\begin{theoremA}
Let $\mathbf{G}$ be a connected reductive group with connected centre defined over $\overline{\mathbb{F}_q}$ for some prime power $q$ and $F$ a Steinberg endomorphism. Let $m$ be the exponent of ${\mathbf{G}}^{F}$ and $\varepsilon_m$ a primitive $m$-th root of unity. For $\sigma \in \Gal(\mathbb{Q}(\varepsilon_m)/\mathbb{Q})$, let $b \in \mathbb{Z}$ with $\gcd(b,m)=1$ such that $\varepsilon_m^{\sigma} =\varepsilon_m^b$.
If  $\chi \in \Irr({\mathbf{G}}^{F})$ has Jordan decomposition $(s, \nu)$ for $s \in \mathbf{G}^F$ semisimple and $\nu$ a unipotent character of $C_{\mathbf{G}^F}(s)$, then $\chi^{\sigma}$ has Jordan decomposition $(s^b, \nu^\sigma)$.
\end{theoremA}

We chose to study the Suzuki and Ree groups since they are often excluded in statements about groups of Lie type and thus require a special treatment. For the other groups in Theorem \ref{thmA}, we cannot systematically choose a local group in the same way as for other groups of Lie type, see Section \ref{sectionnonconvenient}. Therefore, they have to be considered separately.

\section{Preliminaries}
We introduce the notation we use throughout this paper. Let $G$ be a finite group and $p$ a prime dividing $|G|$. 
\subsection{Actions on characters}
The set of irreducible characters of $G$ is denoted by $\Irr(G)$ and the set of irreducible $p'$-characters by $\Irr_{p'}(G)$. 
We consider the subgroup $\mathcal{H} \subseteq \Gal(\mathbb{Q}^{\ab}/\mathbb{Q})$ consisting of those Galois automorphisms that map every root of unity $\xi$ with $\ord(\xi)$ prime to $p$ to $\xi^{p^k}$ for some fixed integer $k$. The group $\mathcal{H}$ acts on $\Irr(G)$ by first evaluating the character and then applying the Galois automorphisms to the character values. 

We denote the automorphism group of $G$ by $\Aut(G)$ and the inner automorphism group by $\Inn(G)$. If $N$ is a normal subgroup of $G$, we can define an action of the setwise stabilizer $\Aut(G)_N$ on a character $\psi$ of $N$ via $\psi^f(x)=\psi(x^{f^{-1}})$ for all $x \in N$, $f \in \Aut(G)_N$. If $f$ is an inner automorphism, we also use the group element itself to denote conjugation with $g \in G$. 

\subsection{Inductive McKay--Navarro condition}
We want to show that a simple group $G$ satisfies the inductive McKay--Navarro (iMcKN) condition for $p$ from \cite[Definition 3.5]{navarro2019reduction}. 
Since the condition is quite technical, we only state a simplified version that implies the iMcKN condition for groups with trivial Schur multiplier and cyclic outer automorphism group.
\begin{condition}\label{imck}
Let $G$ be a non-abelian simple group with cyclic outer automorphism group and trivial Schur multiplier. We let $R \in \Syl_p(G)$ and write $\Gamma :=\Aut(G)_R :=\{f \in \Aut(G) \mid f(R) = R\}$. Then $G$ satisfies the iMcKN condition for $p$ if the following holds for all $\psi \in {\Irr}_{p'}(G)$:
\begin{enumerate}[(1)]
\item There exists a $\Gamma$-stable subgroup $N_G(R) \subseteq N \subsetneq G$ and a $\Gamma \times \mathcal{H}$-equivariant bijection $$\Omega: {\Irr}_{p'}(G) \rightarrow {\Irr}_{p'}(N).$$
\item There are $(\Gamma \times \mathcal{H})_\psi$-invariant extensions $\hat{\psi} \in \Irr(G \rtimes \Gamma_\psi)$ of $\psi$ and $\widehat{\Omega(\psi)} \in \Irr(N \rtimes \Gamma_\psi)$ of $\Omega(\psi)$.
\end{enumerate}
\end{condition}
As in \cite[Proposition 2.3]{spaeth2012mckaydefining} it is clear that this implies the inductive McKay--Navarro condition from \cite[Definition 3.5]{navarro2019reduction} for $G$ and $p$.

\subsection{Suzuki and Ree groups}
In Section 3 and Section 4, let $G \in \{{}^2\mathsf{B}_2(2^{2f+1}),{}^2\mathsf{G}_2(3^{2f+1}), ^2\mathsf{F}_4(2^{2f+1})\}$ be a Suzuki or Ree group for $f \geq 0$ and $l$ its defining characteristic. By $p$ we denote a prime dividing $|G|$ different from $l$. 
The Suzuki and Ree groups are all simple and have trivial Schur multiplier except for ${}^2\mathsf{B}_2(2), {}^2\mathsf{B}_2(8), {}^2\mathsf{G}_2(3),$ and ${}^2\mathsf{F}_4(2)$ that will be treated seperately in Section \ref{sectionq2=2}. Thus, it suffices to verify Condition \ref{imck}.

Let $\mathbf{G}$ be the algebraic group of type $\mathsf{B}_2,$ $\mathsf{F}_4$, or $\mathsf{G}_2$, respectively, defined over an algebraic closure of $\mathbb{F}_l$. We denote by $F_l$ the standard field endomorphism of $\mathbf{G}$ and by $\gamma$ the graph morphism of $\mathbf{G}$ induced by the exceptional symmetry of the Dynkin diagram. Setting $F=F_l^f \circ \gamma$, we obtain $\mathbf{G}^F \cong G$.

\section{Equivariant bijections}
We first prove the existence of a group $N$ and a $\Gamma \times \mathcal{H}$-equivariant bijection as in (1) of Condition \ref{imck}. 
We already know that ${\mathbf{G}}^{F}$ satisfies the inductive McKay-condition for all $p \in \mathbb{P}$ by \cite{isaacsmallenavarro2007} and \cite{cabanesspaeth2013}. Thus, there exists a group $N$ and a $\Gamma$-equivariant bijection $$\Omega_{\text{iMcK}}: \Irr_{p'}({\mathbf{G}}^{F}) \rightarrow \Irr_{p'}(N).$$
We want to show that this bijection can also be chosen $\mathcal{H}$-equivariant.

We know that $(\mathbf{G},F)$ is selfdual by \cite[p.120]{carter1985finite} and the centre of $\mathbf{G}$ is trivial by \cite[Theorem 1.12.5]{gorensteinlyonssolomon1998}. For a semisimple element $s \in {\mathbf{G}}^{F}$ we denote by $\mathcal{E}(\mathbf{G}^F, s)$ the set of irreducible constituents of all Deligne--Lusztig generalized characters $R_{\mathbf{T}}^\mathbf{G}(s)$ for $F$-stable maximal tori $\mathbf{T} \subseteq \mathbf{G}$ with $s \in \mathbf{T}^F$. By the Jordan decomposition of irreducible characters of $\mathbf{G}^F$, we have a partition $$\Irr(\mathbf{G}^F)=\bigcup_{s}\mathcal{E}(\mathbf{G}^F, s)$$ where $s$ ranges over the conjugacy classes of semisimple elements of $\mathbf{G}^F$ \cite[Theorem 2.6.2]{geckmalle2020}.

The outer automorphism group of $\mathbf{G}^F$ is then $\langle \gamma \rangle=\langle F_l \rangle$ and has order $2f+1$ \cite[Theorem 2.5.12]{gorensteinlyonssolomon1998}. We already know that we have
$\mathcal{E}(\mathbf{G}^F, s)^\kappa=\mathcal{E}(\mathbf{G}^F, s^\kappa)$
for every outer automorphism $\kappa$ of $\mathbf{G}^F$ \cite[Lemma 3.2]{johansson2020}. 

\subsection{Equivariance of the Jordan decomposition}
In the following, we need a $\Gamma \times \mathcal{H}$-equivariant Jordan decomposition for the irreducible characters of $\mathbf{G}^F$.
Therefore, we first prove that \cite[Theorem 5.1]{srinivasanvinroot2019} continues to hold for the Suzuki and Ree groups.

\begin{prop} \label{propjordan}
Let $m$ be the exponent of ${\mathbf{G}}^{F}$ and $\varepsilon_m$ a primitive $m$-th root of unity. Let $\sigma \in \Gal(\mathbb{Q}(\varepsilon_m)/\mathbb{Q})$ and $b \in \mathbb{Z}$ with $\gcd(b,m)=1$ such that $\varepsilon_m^{\sigma} =\varepsilon_m^b$. 
Let $\chi \in \Irr({\mathbf{G}}^{F})$ with Jordan decomposition $(s, \nu)$ for $s \in \mathbf{G}^F$ semisimple and $\nu$ a unipotent character of $C_{\mathbf{G}^F}(s)$. Then, $\chi^{\sigma}$ has Jordan decomposition $(s^b, \nu^\sigma)$.
\end{prop}
\begin{proof}
We know by \cite[Proposition 3.3.16]{geckmalle2020} that  
$$\mathcal{E}(\mathbf{G}^F, s)^\sigma=\mathcal{E}(\mathbf{G}^F, s^b).$$
If $\chi$ is a unipotent character, its Jordan decomposition is $(1,\chi)$. Then $\chi^{\sigma}$ is also unipotent and its Jordan decomposition is $(1,\chi^\sigma)$ as claimed. 

Assume that the characters in $\mathcal{E}(\mathbf{G}^F, s)$ have pairwise distinct degrees. The degrees of the characters in $\mathcal{E}(\mathbf{G}^F, s)$ correspond uniquely to the degrees of the unipotent characters of $C_{\mathbf{G}^F}(s)$ by \cite[Corollary 2.6.6]{geckmalle2020}. It follows that $\nu$ is fixed by $\sigma$ and, if $\chi \in \mathcal{E}(\mathbf{G}^F, s)$ has Jordan decomposition $(s, \nu)$, the character $\chi^\sigma$ has Jordan decomposition $(s^b, \nu)$.

It remains to consider the characters that are not unipotent and cannot be distinguished in their rational series by their character degrees. 
We identify and study these cases by looking at the known generic character tables (see e.g. \cite{chevie}).

If ${\mathbf{G}}^{F}$ is of type ${}^2\mathsf{B}_2$ or ${}^2\mathsf{G}_2$, no character like this exists. 
For Ree groups of type ${}^2\mathsf{F}_4$, we use the notation of characters as in \cite{chevie} and denote the characters of type $i$ by $\chi_i(k)$. We do not specify the parameter $k$ here but assume it to be as described in \cite{chevie}. The only characters that still have to be considered are
$$\mathcal{E}(\mathbf{G}^F, s_2(k))=\{\chi_{22}(k), \chi_{23}(k), \chi_{24}(k), \chi_{25}(k)\}$$ where we denote the semisimple elements $s_i(k)$ as labeled in \cite{chevie}.
As before, it is clear that the claim holds for $\chi_{22}(k)$ and $\chi_{25}(k)$ since they are determined by their degrees.
By looking at the explicit character values given in \cite{chevie} we see
$$\chi_{23}(k)^\sigma=\left\{ \begin{array}{ll}
\chi_{23}(bk) & \text{if } i^\sigma=i, \\
\chi_{24}(bk) & \text{if } i^\sigma=-i,
\end{array} \right. \quad \chi_{24}(k)^\sigma=\left\{ \begin{array}{ll}
\chi_{24}(bk) & \text{if } i^\sigma=i, \\
\chi_{23}(bk) & \text{if } i^\sigma=-i,
\end{array} \right. $$
where $i$ denotes a primitive fourth root of unity.

In the notation of \cite{shinoda1975conj2f4} and using the Steinberg presentation, the semisimple elements $s_2(k)$ correspond to elements conjugate to $t_1(k)=\left(1,1,\varepsilon_{q^2-1}^k, \varepsilon_{q^2-1}^{k(2^{f+1}-1)}\right)$ with $\varepsilon_{q^2-1}$ a primitive $(q^2-1)$-th root of unity. 
We see that $s_2(k)^b=s_2(bk)$. 

Now, the centralizers of $s_2(k)$ are of type ${}^2\mathsf{B}_2$ and thus have four unipotent characters $1, \mathcal{W}, \overline{\mathcal{W}}, \St$ as given in \cite[Theorem 4.6.9]{geck2003introduction}. We have 
$$1^\sigma=1, \quad \mathcal{W}^\sigma=\left\{ \begin{array}{ll}
\mathcal{W} & \text{if } i^\sigma=i, \\
\overline{\mathcal{W}} & \text{if } i^\sigma=-i,
\end{array} \right. \quad \overline{\mathcal{W}}^\sigma=\left\{ \begin{array}{ll}
\overline{\mathcal{W}} & \text{if } i^\sigma=i, \\
\mathcal{W} & \text{if } i^\sigma=-i,
\end{array} \right. \quad \St^\sigma=\St.$$
This shows the claim for groups of type ${}^2\mathsf{F}_4$.
\end{proof}
\subsection{$\mathcal{H}$-equivariance of character bijections} \label{sectionbijection} We want to show that the parametrization of irreducible $p'$-characters of $\mathbf{G}^F$ from \cite[Theorem 8.5]{malle2008height0} leads to a $\Gamma \times \mathcal{H}$-equivariant bijection. We already know by \cite[Theorem 6.1]{cabanesspaeth2013} that we have $\Gamma$-equivariance for type ${}^2\mathsf{F}_4$. For the other types, this can be shown in the same way or by looking at the explicit bijections in \cite[Section 16 and 17]{isaacsmallenavarro2007} that coincide with Malle's parametrization. 

First, we prove a number theoretical lemma about the prime divisors of some cyclotomic polynomials.
\begin{lemma} \label{lemmamodulo}
Let $q^2=2^{2f+1}$ for an integer $f \geq 1$ and let $p$ be an odd prime.
\begin{enumerate}[(a)]
\item If $p \mid q^2-1$, then $p \equiv 1,7 \mod 8$.
\item If $p \mid q^4-q^2+1$ and $p \neq 3$, then $p \equiv 1 \mod 3$.
\item If $p \mid q^4 \pm \sqrt{2}q^3 +q^2 \pm \sqrt{2}q +1$ and $p \neq 3$, then $p \equiv 1,11 \mod 12$.
\item If $p \mid q^2 \pm \sqrt{2}q +1$, then $p \equiv 1 \mod 4$.
\end{enumerate}
\end{lemma}
\begin{proof} 
We know by Fermat's little theorem that we have $2^{p-1} \equiv 1 \mod p$.
\begin{enumerate}[(a)]
\item We have $q^2-1=2^{2f+1}-1 \equiv 0 \mod p$, thus $2^{2f+1} \equiv 1 \mod p$. It follows $$2^{\gcd(2f+1,p-1)}=2^{\gcd(2f+1,\frac{p-1}{2})} \equiv 1 \mod p$$ and since $p$ is odd further $2^{(p-1)/2} \equiv 1 \mod p$. The claim follows from the second supplement to quadratic reciprocity.
\item Assume that we have additionally $p \mid q^4-1$. Then it follows $p \mid q^2+1$ or $p \mid q^2-1$. Now $p \mid q^4-q^2+1$ implies $p \mid q^2-2$ which is not possible at the same time for $p \neq 3$. Thus, we know $p \nmid q^4-1$. 
Since $p \mid q^4-q^2+1$, it follows $p \mid q^6+1=(q^4-q^2+1)(q^2+1)$ and $2^{3(2f+1)} \equiv -1 \mod p$. Together with Fermat's little theorem we have $2^{\gcd(p-1,6(2f+1))}\equiv 1 \mod p$. If $p \equiv 2 \mod 3$, it follows $\gcd(p-1,6(2f+1))=\gcd(p-1,2(2f+1))$ which implies $2^{2(2f+1)}=q^4 \equiv 1 \mod p$ and $p \mid q^4-1$. This is a contradiction and since $p \neq 3$, the claim holds.
\item We know $p \mid ( q^4 + \sqrt{2}q^3 +q^2 + \sqrt{2}q +1)( q^4 - \sqrt{2}q^3 +q^2 - \sqrt{2}q +1)(q^4+1)=q^{12}+1$ and as before it follows $$(*) \quad 2^{\gcd(6(2f+1),p-1)} \equiv -1 \mod p, \quad  2^{\gcd(12(2f+1),p-1)} \equiv 1 \mod p.$$
Since $p \in \mathbb{P}$, only $p \equiv 1,5,7,11 \mod 12$ are possible. If we have $p \equiv 5 \mod 12$, then it follows 
$2^{4\gcd(2f+1,\frac{p-1}{4})} \equiv 1 \mod p$. If we have $p \equiv 7 \mod 12$, then it follows $2^{6\gcd(2f+1,\frac{p-1}{6})} \equiv 1 \mod p$. These equations contradict $(*)$, thus we have shown the claim. 
\item Let $p \mid (q^2 + \sqrt{2}q +1)(q^2 - \sqrt{2}q +1)=q^4+1$. Then it follows $2^{2(2f+1)}\equiv -1 \mod p$. If we have $p \equiv 3 \mod 4$, then we get as before $2^{2\gcd(2f+1, \frac{p-1}{2})}\equiv 1 \mod p$ which leads again to a contradiction. The claim follows.
\end{enumerate}
\end{proof}

In the following, we use the notion of Sylow $\Phi^{(p)}$-tori as defined in \cite[Section 8.1]{malle2008height0}. Thereby, let $\Phi^{(p)}$ be the cyclotomic polynomial over $\mathbb{Q}(\sqrt{l})$ that divides the generic order of $\mathbf{G}^F$ and satisfies $p \mid \Phi^{(p)}$. Then, $\Phi^{(p)}$ is well-defined except for $p=2$ and ${}^2 \mathsf{G}_2(q^2)$, or $p=3$ and ${}^2 \mathsf{F}_4(q^2)$. If $\mathbf{G}^F={}^2 \mathsf{F}_4(q^2)$ and $q^2 \equiv 2,5 \mod 9$, we set $\Phi^{(3)}:=q^2+1$. We exclude the other cases and assume $p\neq 3$ if $q^2 \equiv 2,5 \mod 9$ and $\mathbf{G}^F={}^2 \mathsf{F}_4(q^2)$, and $p\neq 2$ if $\mathbf{G}^F={}^2 \mathsf{G}_2(q^2).$
These groups and primes will be considered separately in Proposition \ref{propspecialg2f4}.

For Frobenius endomorphisms of connected reductive groups, integers $d \geq 1$, and the corresponding cyclotomic polynomial $\Phi_d$ over $\mathbb{Z}$, Sylow $\Phi_d$-tori can be defined analogously to the Sylow $\Phi^{(p)}$-tori from above, see e.g. \cite[3.5.1]{geckmalle2020}.

Let $Q$ be a Sylow $p$-subgroup of $\mathbf{G}^F$. Then, there exists a Sylow $\Phi^{(p)}$-torus $\mathbf{S}$ such that $N_{\mathbf{G}^F}(Q) \leq N_{\mathbf{G}^F}(\mathbf{S})$ by \cite[Theorem 8.4]{malle2008height0}.
Let $\mathbf{T}=C_{\mathbf{G}}(\mathbf{S})$ be its centralizer in $\mathbf{G}$. We know from the proof of \cite[Lemma 6.5]{spaeth2009mckayexceptional} that $\mathbf{T}$ is a maximal torus.

If $\mathbf{M}$ is an $F$-stable connected reductive group and $\mathbf{M}_0 \leq \mathbf{M}$ is an $F$-stable Levi subgroup, we denote by $\W_{\mathbf{M}}(\mathbf{M}_0)^F:=N_{\mathbf{M}}(\mathbf{M}_0)^F/\mathbf{M}_0^F$ the so-called \textit{relative Weyl group}. 
If $\mathbf{M}_0$ is even the centralizer of a Sylow $\Phi_d$-torus, it is called a $d$-split Levi subgroup. For $\lambda \in \mathcal{E}(\mathbf{M}_0,1)$ such that ${}^*R_{\mathbf{L}_0}^{\mathbf{M}_0}(\lambda) =0$ for any $d$-split $\mathbf{L}_0 < \mathbf{M}_0$, we set
$$\mathcal{E}(\mathbf{M}^F,\mathbf{M}_0,\lambda):=\{\psi \in \Irr(\mathbf{M}^F) \mid \psi \textrm{ is a constituent of } R_{\mathbf{M}_0}^{\mathbf{M}}(\lambda) \}.$$

\begin{prop} \label{proprelweylbijection}
Let $\mathbf{T}$ be as above. For a semisimple element $s \in \mathbf{T}^F$, we set $\mathbf{H}=C_{\mathbf{G}}(s)$. Then there exists a $\Gamma \times \mathcal{H}$-equivariant bijection $$\mathcal{I}: \Irr(\W_{\mathbf{H}^F}(\mathbf{T})) \rightarrow \mathcal{E}(\mathbf{H}^F, \mathbf{T}, 1).$$ 
\end{prop}
\begin{proof}
We first consider ${}^2\mathsf{F}_4(q^2)$.
As in \cite[Proof of Theorem 6.1]{cabanesspaeth2013}, we know that $(\mathbf{H},F)$ is of type ${}^2\mathsf{F}_4$, ${}^2\mathsf{B}_2$, ${}^2\mathsf{A}_2$, $\mathsf{A}_1$ or $\mathsf{A}_0$. From there we also know that the action of $\Gamma$ on all relevant characters is trivial and that there exists a bijection between the character sets.

We know from the proof of Proposition \ref{propjordan} that $\mathcal{H}$ acts trivially on all unipotent characters except possibly on those of ${}^2\mathsf{F}_4$ and ${}^2\mathsf{B}_2$.

If $(\mathbf{H},F)$ is of type $\mathsf{A}_0$, the corresponding relative Weyl group is trivial and $\mathcal{H}$ acts trivially on its unique irreducible character.

If $(\mathbf{H},F)$ is of type $\mathsf{A}_1$, all maximal $F$-stable tori $\mathbf{T}$ are minimal $d$-split Levi subgroups for a $d \in \{1,2\}$. The structure of $\W_\mathbf{H}(\mathbf{T})^F$ is described in \cite[Example 3.5.14(b)]{geckmalle2020} and we see $|\W_\mathbf{H}(\mathbf{T})^F|=2$. Thus, its characters have integer character values and $\mathcal{H}$ acts trivially.

If $(\mathbf{H},F)$ is of type ${}^2\mathsf{A}_2$, all maximal $F$-stable tori $\mathbf{T}$ are minimal $d$-split Levi subgroups for a $d \in \{1,2,6\}$. The structure of $\W_\mathbf{H}(\mathbf{T})^F$ is described in \cite[Example 3.5.14 (c)]{geckmalle2020}. For $d=1$ and $d=2$, we have $|\W_\mathbf{H}(\mathbf{T})^F|=2$ and $\W_\mathbf{H}(\mathbf{T})^F \cong \Sym(3)$, respectively. Thus, their characters have integer values and $\mathcal{H}$ acts trivially. 
If $d=6$, then $\W_\mathbf{H}(\mathbf{T})^F \cong C_3$ and $\mathbf{S}$ is by definition a Sylow $6$-torus of $\mathbf{H}$ with $|\mathbf{S}^F|=q^4-q^2+1$. Thus, we have $p \mid q^4-q^2+1$ and by Lemma \ref{lemmamodulo} we know $p \equiv 1 \mod 3$. It follows that $\mathcal{H}$ acts trivially on all characters of $\W_\mathbf{H}(\mathbf{T})^F$.

If $(\mathbf{H},F)$ is of type ${}^2\mathsf{B}_2$, the structure of all possible maximal tori $\mathbf{T}$ is given in \cite[Table 4.4]{geck2003introduction}. We use the notation introduced in \cite[Section 4.6]{geck2003introduction}. If $\mathbf{T}=T_1$, then we have $\W_\mathbf{H}(\mathbf{T})^F \cong C_2$ and $\mathcal{E}(\mathbf{H}^F, \mathbf{T}, 1)=\{1_{\mathbf{G}^F}, \St_{\mathbf{G}^F}\}$ by \cite[Proposition 4.6.7]{geck2003introduction}. Thus, $\mathcal{H}$ acts trivially on all occurring characters.
If $\mathbf{T}=T_{s_\alpha s_\beta s_\alpha}$ or $\mathbf{T}=T_{s_\alpha}$, then $\W_\mathbf{H}(\mathbf{T})^F \cong C_4$ and $\mathcal{E}(\mathbf{H}^F, \mathbf{T}, 1)=\{1_{\mathbf{G}^F}, \mathcal{W}, \overline{\mathcal{W}}, \St_{\mathbf{G}^F}\}$. As in the proof of Proposition \ref{propjordan} we see that $\sigma \in \mathcal{H}$ acts trivially on both sets if the imaginary unit $i$ is fixed and interchanges two characters otherwise. Therefore, we find a bijection as claimed.

Assume that $(\mathbf{H},F)$ is of type ${}^2\mathsf{F}_4$. The maximal $F$-stable tori of $\mathbf{G}$ are described in \cite{shinoda1975conj2f4}. We use the same notation and denote the corresponding fixed point sets by $T(1),\ldots, T(11)$. Since $\mathbf{G}^F$ has order $$q^{24} ( q^4 + \sqrt{2}q^3 +q^2 + \sqrt{2}q +1)( q^4 - \sqrt{2}q^3 +q^2 - \sqrt{2}q +1)(q^2 - \sqrt{2}q +1)^2(q^2 + \sqrt{2}q +1)^2(q^2-1)^2(q^2+1)^2(q^4-q^2+1),$$
we have the following Sylow $\Phi^{(p)}$-tori:
$$p \mid q^2 -1 \textrm{ and } \mathbf{T}^F \cong T(1),$$
$$p \mid q^2 - \sqrt{2}q +1 \textrm{ and } \mathbf{T}^F \cong T(6), $$
$$p \mid q^2 + \sqrt{2}q +1 \textrm{ and } \mathbf{T}^F \cong T(7), $$
$$p \mid q^2 +1 \textrm{ and } \mathbf{T}^F \cong T(8),$$
$$p \mid q^4 -q^2 +1 \textrm{ and } \mathbf{T}^F \cong T(9), $$
$$p \neq 3, \; p \mid q^4 - \sqrt{2}q^3 +q^2 - \sqrt{2}q +1 \textrm{ and } \mathbf{T}^F \cong T(10),  \textrm{ or}$$
$$p \neq 3, \; p \mid q^4 + \sqrt{2}q^3 +q^2 + \sqrt{2}q +1 \textrm{ and } \mathbf{T}^F \cong T(11).$$
The $21$ unipotent characters of $\mathbf{G}^F$ are described in \cite{malle1990unip2f4}. We denote them by $\chi_1, \ldots , \chi_{21}$ and enumerate them as in the database of \textsf{CHEVIE}. By explicitly studying the character values we see that $\sigma$ acts on the characters by permuting the indices as 
\begin{center}
\begin{tabular}{c @{ if }c}
$(15,16) $ & $ \varepsilon_{12}^\sigma=\varepsilon_{12}^5$, \\
$(2,3)(11,12)(13,14)(19,20) $ & $ \varepsilon_{12}^\sigma=\varepsilon_{12}^7$, \\
$(2,3)(11,12)(13,14)(15,16)(19,20) $ & $ \varepsilon_{12}^\sigma=\varepsilon_{12}^{11}$.
\end{tabular}
\end{center}
For every possible Sylow torus $\mathbf{T}$, the elements of $\mathcal{E}(\mathbf{H}^F, \mathbf{T}, 1)$ are identified in \cite[p.374]{lusztig1984charofreductivegroups}. The corresponding relative Weyl group can be determined using the data in \cite{shinoda1975conj2f4} or is given in \cite{brouemallemichel1993}.

First, let $p \mid q^2 -1$. Then, the relative Weyl group is isomorphic to the dihedral group of order $16$. Constructing its characters we see by using Lemma \ref{lemmamodulo}(a) that $\mathcal{H}$ acts trivially on them. The same is true for the characters in $\mathcal{E}(\mathbf{H}^F, \mathbf{T}, 1)= \{\chi_1, \chi_4, \chi_5, \chi_6, \chi_7, \chi_{18}, \chi_{21}\}.$

For $p \mid q^2 \pm \sqrt{2}q +1$, the relative Weyl group is isomorphic to the complex reflection group with Shepard--Todd number $8$ and has the group structure $C_4.\Sym(4)$. We can compute its character table with \textsf{GAP} and see with Lemma \ref{lemmamodulo}(d) that $\mathcal{H}$ acts trivially on its characters. Since $p \equiv 1 \mod 4$,  it is easy to see that the same is true for the characters in $\mathcal{E}(\mathbf{H}^F, \mathbf{T}, 1)$.

For $p \mid q^2 +1$, the relative Weyl group is isomorphic to $\GL_2(3)$. Constructing its characters we see that $\mathcal{H}$ acts trivially on them and the same is true for $\mathcal{E}(\mathbf{H}^F, \mathbf{T}, 1)$.

For $p \mid q^4-q^2 +1$, the relative Weyl group is isomorphic to $C_6$. Since $p \equiv 1 \mod 3$ by Lemma \ref{lemmamodulo}(b), we see that $\mathcal{H}$ acts trivially on the characters of the Weyl group and on $\mathcal{E}(\mathbf{H}^F, \mathbf{T}, 1)$.

For $p \mid q^4 \pm \sqrt{2}q^3+q^2 \pm \sqrt{2}q +1$, the relative Weyl group is isomorphic to $C_{12}$. Since $p \equiv 1,11 \mod 12$ by Lemma \ref{lemmamodulo}(c), it is clear that $\mathcal{H}$ acts trivially on the characters of the Weyl group and on $\mathcal{E}(\mathbf{H}^F, \mathbf{T}, 1)$.
With this we have shown the claim for ${}^2\mathsf{F}_4(q^2)$. 

We now study ${}^2\mathsf{B}_2(q^2)$. Then, we know from \cite[p.52]{deriziotisliebeck1985centralizerstwisted} that $(\mathbf{H}, F)$ is of type $\mathsf{A}_0$ or ${}^2\mathsf{B}_2$. Thus, we have already shown the existence of a suitable bijection before.

For ${}^2\mathsf{G}_2(q^2)$, $(\mathbf{H}, F)$ is of type $\mathsf{A}_0$, $\mathsf{A}_1$  or ${}^2\mathsf{G}_2$ \cite[p.52]{deriziotisliebeck1985centralizerstwisted}. The cases $\mathsf{A}_0$ and $\mathsf{A}_1$ can be treated as above. If $\mathbf{H}$ is of type ${}^2\mathsf{G}_2(q^2)$, there are four different types of tori $\mathbf{T}$. The respective sets $\mathcal{E}(\mathbf{H}^F, \mathbf{T}, 1)$ are given in \cite[p.376]{lusztig1984charofreductivegroups}. If $\mathbf{T}^F$ is of order $q^2-1$, there are two characters in $\mathcal{E}(\mathbf{H}^F, \mathbf{T}, 1)$ and they are both $\mathcal{H}$-invariant. For all other tori, there are six characters in $\mathcal{E}(\mathbf{H}^F, \mathbf{T}, 1)$ and $\sigma \in \mathcal{H}$ induces a double transposition if a third root of unity is not preserved by $\sigma$ and fixes them otherwise. 

The relative Weyl group is cyclic by \cite[Table 3]{brouemallemichel1993} and thus isomorphic to $C_2$ or $C_6$. Thus it is easy to see that the action of $\mathcal{H}$ on the characters of the Weyl group is the same as on the respective unipotent characters. This proves the claim.
\end{proof}

\begin{lemma} \label{lemmaextensionmap}
For $\mathbf{T}$ as defined above, there exists a $\Gamma \times \mathcal{H}$-equivariant extension map for $\mathbf{T}^F \triangleleft N_{\mathbf{G}}(\mathbf{T})^F$ as in \cite[Definition 2.9]{cabanesspaeth2013}.
\end{lemma}
\begin{proof}
If $q$ is even, this follows directly from the construction in \cite[Lemma 4.2]{spaeth2009mckayexceptional} together with the existence of a very good Sylow twist \cite[Theorem D]{spaeth2009mckayexceptional} and the fact that $H$ as in \cite[Setting 2.1]{spaeth2009mckayexceptional} is trivial.

For type ${}^2\mathsf{G}_2$, we have $ N_{\mathbf{G}^F}(\mathbf{T})=\mathbf{T}^F \rtimes C$ for some cyclic group $C$ (see proof of Proposition \ref{proprelweylbijection}, \cite[Theorem C]{kleidman1988maxsubgroups} and \cite{levchuknuzhin1985}). 
Every irreducible character of $\mathbf{T}^F$ is linear and we can extend it trivially to its inertia subgroup in $ N_{\mathbf{G}^F}(\mathbf{T})$. It is clear that this extension is $\Gamma \times \mathcal{H}$-equivariant.
\end{proof}
\begin{rem}
Note that, for $s \in \mathbf{T}^F$ semisimple and $k \in \mathbb{Z}$ coprime to $\ord(s)$, the maps $i_{s,1}$ and $i_{s^k, 1}$ constructed in \cite[Proof of Corollary 3.3]{cabanesspaeth2013} are the same. The equality of domain and codomain follows directly from the coprimeness. We see that the maps are equal since for $t \in \mathbf{T}$ with $s^t=s'$ we also have $(s^k)^t=(s')^k$.
\end{rem}
 
\begin{prop} \label{propequivbijection}
With $\mathbf{T}$ defined as above, there exists a $\Gamma \times \mathcal{H}$-equivariant bijection $$\Omega: \Irr_{p'}(N_{\mathbf{G}^F}(\mathbf{T})) \rightarrow  \Irr_{p'}(\mathbf{G}^F).$$
\end{prop}
\begin{proof}
We follow the proofs of \cite[Theorem 4.5, Theorem 6.1]{cabanesspaeth2013}. The $p'$-characters of $\mathbf{G}^F$ and $N:=N_{\mathbf{G}^F}(\mathbf{T})=N_{\mathbf{G}^F}(\mathbf{S})$ can be parametrized by triples $(s, \lambda, \eta)$ with $s \in \mathbf{T}^F$ semisimple, $\lambda \in \mathcal{E}(\mathbf{T}^F,1) $ a $p'$-character, and $\eta \in \Irr_{p'}(\W_{\mathbf{H}}(\mathbf{T}, \lambda)^F)$  for $\mathbf{H}=C_{\mathbf{G}}(s)$ \cite[Theorem 8.5]{malle2008height0}. 
Since $\mathbf{T}^F$ is a torus, $\lambda$ is the trivial character and we set $$\mathcal{M}=\{(s,\eta) \mid s \in \mathbf{T}^F \text{ semisimple up to } N \text{-conjugacy}, \eta \in \Irr_{p'}(\W_{\mathbf{H}}(\mathbf{T})^F) \}.$$ Thus, as in \cite[Section 4]{cabanesspaeth2013} there are bijections 
$$\psi^{(N)}: \mathcal{M} \rightarrow \Irr_{p'}(N_{\mathbf{G}^F}(\mathbf{T})), \quad \psi^{(G)}: \mathcal{M} \rightarrow  \Irr_{p'}(\mathbf{G}^F). $$ 
As in the proof of \cite[Theorem 4.5]{cabanesspaeth2013}, we want to show that these bijections can be chosen such that
$$\left(\psi^{(N)}(s,\eta)^{\gamma}\right)^\sigma =  \psi^{(N)}\left(\gamma(s^k),(\eta^\gamma)^\sigma\right), \quad  \left(\psi^{(G)}(s,\eta)^\gamma\right)^\sigma =  \psi^{(G)}\left(\gamma(s^k),(\eta^\gamma)^\sigma\right)$$
for $\gamma \in \Gamma$ and any $\sigma \in \mathcal{H}$ such that an $\exp(\mathbf{G}^F)$-th root of unit is mapped to its $k$-th power. We already know that $\psi^{(N)}$ and $\psi^{(G)}$ are $\Gamma$-equivariant from \cite[Theorem 6.2]{cabanesspaeth2013} and thus only have to verify the $\mathcal{H}$-equivariance.
We denote by $\chi_s^{\mathbf{\mathbf{T}}}$ the irreducible character of $\mathbf{T}^F$ with Jordan decomposition $(s,1)$ and by $\chi_{s,\nu}^{\mathbf{\mathbf{G}}}$ the one of $\mathbf{G}^F$ with Jordan decomposition $(s,\nu)$. First note that the Jordan decomposition can be chosen $\mathcal{H}$-equivariant by Proposition \ref{propjordan}. 

Denote by $\Lambda: \Irr(\mathbf{T}^F) \rightarrow \bigcup_{\mathbf{T}^F \leq I \leq N_{\mathbf{G}^F}(\mathbf{T})}\Irr(I)$ the extension map from Lemma \ref{lemmaextensionmap}. 
Then we have
\begin{align*}
\psi^{(N)}(s, \eta)^\sigma &= \Ind_{N_{\chi_s^{{\mathbf{T}}}}}^N \left(\Lambda(\chi_s^{{\mathbf{T}}})(\eta \circ i_{s,1})\right)^\sigma 
=\Ind_{N_{\chi_s^{{\mathbf{T}}}}}^N \left(\Lambda(\chi_s^{{\mathbf{T}}})^\sigma  (\eta \circ i_{s,1})^\sigma\right) \\
&= \Ind_{N_{\chi_s^{{\mathbf{T}}}}}^N \left(\Lambda((\chi_s^{{\mathbf{T}}})^\sigma)(\eta^\sigma \circ i_{s,1})\right) 
= \Ind_{N_{\chi_{s^k}^{\mathbf{{T}}}}}^N \left(\Lambda(\chi_{s^k}^{{\mathbf{T}}}) (\eta^\sigma \circ i_{s^k,1})\right)
= \psi^{(N)}(s^k, \eta^\sigma).
\end{align*}
For the characters of $\mathbf{G}^F$, let $\mathcal{I}$ be the map from \ref{proprelweylbijection}. Then it follows
$$\psi^{(G)}(s,\eta)^\sigma=(\chi_{s,\mathcal{I}(\eta)}^{\mathbf{G}})^\sigma=\chi_{s^k,\mathcal{I}(\eta)^\sigma}^{\mathbf{G}}=\chi_{s^k,\mathcal{I}(\eta^\sigma)}^{\mathbf{G}}=\psi^{(G)}(s^k,\eta^\sigma).$$
This shows the claim.
\end{proof}

\section{Character Extensions}
In this section, we verify (2) of Condition \ref{imck}. We start by recalling some easy observations about character extensions.  
\begin{rem}\label{remextension}
Let $\psi$ be an irreducible character of a finite group $M$ and $A \leq \Aut(M)$.
\begin{enumerate}[(a)]
\item We find an $(A \times \mathcal{H})_\psi$-invariant extension $\hat{\psi}$ of $\psi$ to the inner automorphisms $\{c_g \mid g \in M \} \leq \Aut(M)$ of $M$ by setting $\hat{\psi}(m,c_g):=\psi(m)\psi(g)$. In the following, we use this canonical extension without further notice.
\item \label{remlinear}If $\psi$ is linear, it can be trivially extended to $M \rtimes A_\psi$ by setting $\hat{\psi}(m,a):=\psi(m)$ for all $a \in A_\psi$. This extension is $(A \times \mathcal{H})_\psi$-invariant.
\item \label{reminduction} 
Let $\psi \in \Irr(M)$ and $\tau \in \Irr(X)$ for some $A$-invariant subgroup $X \leq M$. Assume that $\psi$ can be extended to $\hat{\psi} \in \Irr(M \rtimes A_\psi)$, $\langle \psi,\Ind_{X}^M(\tau) \rangle =1$, and that there is an $(A \times \mathcal{H})_\psi$-invariant extension of $\tau$ to $\hat{\tau} \in \Irr(X \rtimes A_\psi)$. Then, there is a unique constituent of $\Ind_{X \rtimes A_\psi}^{M \rtimes A_\psi}(\hat{\tau})$ extending $\psi$  to $X \rtimes A_\psi$ that is $(A \times \mathcal{H})_\psi$-invariant. We know
\begin{align*}
{\Ind}_{X}^{M \rtimes A_\psi}(\tau)=&{\Ind}_{X \rtimes A_\psi}^{M \rtimes A_\psi}\left( \sum_{\beta \in \Irr(A_\psi)}\beta(1) \cdot \beta \, \hat{\tau} \right)=\sum_{\beta \in \Irr(A_\psi)}\beta(1) \cdot \beta \; {\Ind}_{X \rtimes A_\psi}^{M \rtimes A_\psi}(\hat{\tau})\,.
\end{align*}
Now, $\beta' \hat{\psi} \in  \Irr(X \rtimes A_\psi)$ is an extension of $\psi$ for every linear $\beta' \in \Irr(A_\psi)$. By Frobenius reciprocity we have
$$\left\langle \beta' \hat{\psi}, {\Ind}_{X}^{M \rtimes A_\psi}(\tau)  \right\rangle = \left\langle \psi, {\Ind}_{X}^{M }(\tau)  \right\rangle=1. $$ 
Comparing this with the constituents of ${\Ind}_{X}^{M \rtimes A_\psi}(\tau)$ gives us a unique linear $\beta' \in \Irr(A_\psi)$ such that $$\left\langle \beta' \hat{\psi}, {\Ind}_{X\rtimes A_\psi}^{M \rtimes A_\psi}(\hat{\tau})  \right\rangle =1.$$ Since $\hat{\tau}$ is $(A \times \mathcal{H})_\psi$-invariant, it is clear from the induction formula that ${\Ind}_{X\rtimes A_\psi}^{M \rtimes A_\psi}(\hat{\tau})$ is also $(A \times \mathcal{H})_\psi$-invariant. The uniqueness of $\beta' \hat{\psi}$ gives the claim.
\end{enumerate}
\end{rem}  

\begin{lemma}\label{lemmaoddorderreal}
Let $M \unlhd H$ be finite groups such that $M$ has odd index in $H$. If $\chi \in \Irr(M)$ is real, it has a unique extension to $H_\chi$ that is $(H \times \mathcal{H})_\chi$-invariant.
\end{lemma}
\begin{proof}
This can be proven exactly as \cite[Corollary 2.2]{navarrotiep2008rational}. By \cite[Lemma 2.1]{navarrotiep2008rational}, there exists a unique real extension $\tilde{\chi}$ of $\chi$ to $H_\chi$. For $(h,\sigma) \in (H \times \mathcal{H})_\chi$, the character $(\tilde{\chi}^h)^{\sigma}$ is also a real extension of $\chi$ and it follows $(\tilde{\chi}^h)^{\sigma}=\tilde{\chi}$. 
\end{proof}

\subsection{Global characters}

In order to find suitable extensions of the global characters, we want to construct character extensions of some Deligne--Lusztig characters. We use the constructions from \cite[Section 6.3]{ruhstorfer2017navarro}. Let $F_0$ be a Steinberg endomorphisms of $\mathbf{G}$ and $k$ a positive integer. For any $F_0^k$-stable closed subgroup $\mathbf{K}$ of $\mathbf{G}$, we define
$$\underline{\mathbf{K}}=\mathbf{K} \times F_0^{k-1}(\mathbf{K}) \times \ldots \times F_0(\mathbf{K})\leq \mathbf{G}^k=\underline{\mathbf{G}}.$$
We set $$\tau: \underline{\mathbf{G}} \rightarrow \underline{\mathbf{G}}, \quad \tau(g_1,\ldots, g_k)=(g_2, \ldots , g_k,g_1).$$
As in \cite[Section 6.3]{ruhstorfer2017navarro}, $\tau$ is a quasi-central automorphism of $\underline{\mathbf{G}}$ and the projection onto the first coordinate $\pr$ induces isomorphisms 
$$\underline{\mathbf{K}}^{\tau F_0}\cong \mathbf{K}^{F_0^k} , \quad \underline{\mathbf{G}}^{\tau F_0} \rtimes \langle \tau \rangle \cong \mathbf{G}^{F_0^k} \rtimes \langle F_0 \rangle  .$$

We can now apply this construction to find suitable extensions of the global characters.
\begin{prop}\label{propextunip}
Let $\psi$ be a unipotent character of $\mathbf{G}^F$. Then there exists an $\mathcal{H}_\psi$-invariant extension of $\psi$ to $\mathbf{G}^F \rtimes \Gamma$.
\end{prop}
\begin{proof}
We know that $\Gamma$ acts trivially on all unipotent characters of $\mathbf{G}^F$ by \cite[Theorem 2.5]{malle2008extuni}. If $\psi$ is the only character of its degree, it is real and we can extend it to $\langle \gamma \rangle$ as in Lemma \ref{lemmaoddorderreal} and then to the inner automorphisms in $\Gamma$. Otherwise, we can always find an $F$-stable Sylow torus $\mathbf{T}$ of $\mathbf{G}$ such that $\psi$ is a constituent of the Deligne--Lusztig character $\eta:=R_\mathbf{T}^\mathbf{G}(1_{\mathbf{T}^F})$ by the tables in \cite[p.373--376]{lusztig1984charofreductivegroups}. We can choose $\mathbf{T}$ such that all consitutents of $\eta$ have multiplicity $\pm 1$. 

Since all Sylow tori corresponding to the same cyclotomic polynomial are conjugate in $\mathbf{G}^F$, we find a $g \in \mathbf{G}^{F}$ such that $(\gamma \circ c_g)(\mathbf{T})=\mathbf{T}$ where $c_g$ denotes conjugation with $g$. We set $F_0:=\gamma \circ c_g$ and $k:=2f+1$. We know $\gamma^{k}=F$ and  $F_0^{k}=c_{g^*}  \circ F=F \circ c_{g^*}$ with $g^*:=\gamma(g) \cdots \gamma^{k-1}(g) \gamma^{k}(g) \in \mathbf{G}^{F}.$ By the Lang--Steinberg theorem, we find an $h \in \mathbf{G}$ such that $g^*=h^{-1}F_1(h)$. Conjugation with $h$ can be extended to an isomorphism $$\mathbf{G} \rtimes \langle F_0 \rangle \rightarrow \mathbf{G} \rtimes \langle c_h \circ F_0 \circ c_{h^{-1}} \rangle, \quad (x,F_0^i) \mapsto (hxh^{-1}, c_h \circ F_0^i \circ c_{h^{-1}} )$$ that maps $F_0^{k}$-stable elements to $F$-stable elements. Thus, it restricts to an isomorphism $\mathbf{G}^{F_0^{k}} \cong \mathbf{G}^{F}$ and $\mathbf{T}_1:=c_{h^{-1}}(\mathbf{T})$ is an $F_0^k$-stable torus of $\mathbf{G}$.

We can use the construction described above to obtain isomorphisms $\underline{\mathbf{G}}^{\tau F_0} \cong \mathbf{G}^{F_0^k}$, $\underline{\mathbf{T}_1}^{\tau F_0} \cong \mathbf{T}_1^{F_0^k}$ and  maps 
$${\pr}_{\mathbf{G}\rtimes \langle F_0 \rangle }^{\vee}: \mathbb{Z}\Irr(\mathbf{G}^{F_0^k} \rtimes \langle F_0 \rangle) \rightarrow \mathbb{Z}\Irr( \underline{\mathbf{G}}^{\tau F_0} \rtimes \langle \tau \rangle ), \quad {\pr}_{\mathbf{G}}^{\vee}: \mathbb{Z}\Irr(\mathbf{G}^{F_0^k}) \rightarrow \mathbb{Z}\Irr(\underline{\mathbf{G}}^{\tau F_0})$$ 
induced by $\pr$
that restrict to bijections between the respective irreducible characters. 
Let $\underline{\psi}:={\pr}_{\mathbf{G}}^{\vee}\left(\psi^{h^{-1}}\right)$ be the character of $\underline{\mathbf{G}}^{\tau F_0}$ corresponding to $\psi$. 
Setting 
$$\underline{\eta}:=R_{\underline{\mathbf{T}_1}}^{\underline{\mathbf{G}}}\left( 1_{\underline{\mathbf{T}_1}^{F_0^k}} \right)={\pr}_{\mathbf{G}}^{\vee}\left(\eta^{h^{-1}}\right),$$  
we have $\langle \underline{\eta}, \underline{\psi} \rangle =\langle \eta, \psi \rangle $. 
We use the construction of Deligne--Lusztig characters for disconnected groups as in \cite[Definition 2.2]{dignemichel1994grpsnonconn}. As already mentioned in  \cite[Proof of Prop. 3.6]{johansson2020}, most of the constructions and results generalize to our setting of Steinberg endomorphisms. 

Let $\mathbf{B}_1$ be a Borel subgroup of the disconnected group $\underline{\mathbf{G}} \rtimes \langle \tau \rangle$ containing $\underline{\mathbf{T}_1} \rtimes \langle \tau \rangle$.
By \cite[Proposition 1.38]{dignemichel1994grpsnonconn}, we find a $\tau$-stable pair $\widetilde{\mathbf{T}} \subseteq \widetilde{\mathbf{B}}$ that is conjugate to $\underline{\mathbf{T}_1} \rtimes \langle \tau \rangle \subseteq \mathbf{B}_1$ by some element $x \in \underline{\mathbf{G}}^{\tau F_0} \rtimes \langle \tau \rangle$.
Then we can define $$\hat{\underline{\eta}}:= R_{\widetilde{\mathbf{T}}}^{\underline{\mathbf{G}} \rtimes \langle \tau \rangle}\left(1_{\widetilde{\mathbf{T}}^{\tau F_0}}\right)=R_{\underline{\mathbf{T}_1}\rtimes \langle \tau \rangle}^{\underline{\mathbf{G}} \rtimes \langle \tau \rangle}\left(1_{\underline{\mathbf{T}_1}^{\tau F_0}\rtimes \langle \tau \rangle}\right) \in \mathbb{Z}\Irr(\underline{\mathbf{G}}^{\tau F_0} \rtimes \langle \tau \rangle).$$ This character extends $\underline{\eta}$ by \cite[Corollaire 2.4(i)]{dignemichel1994grpsnonconn}. 
We have $$((\underline{\mathbf{G}} \rtimes \langle \tau \rangle)^{\tau})^\circ= (\{(g, \ldots, g) \mid g \in \mathbf{G}\} \rtimes \langle \tau \rangle )^ \circ \cong \mathbf{G}$$ where the isomorphism is given by projection to the first coordinate $\pr$. Note that $\tau F_0$ acts on the group elements as $F_0$. Analogously, we have $\pr ((\widetilde{\mathbf{T}}^\tau)^\circ)= c_x(\mathbf{T}_1)$. It follows
$$  R_{(\widetilde{\mathbf{T}}^\tau)^\circ}^{((\underline{\mathbf{G}} \rtimes \langle \tau \rangle)^{\tau})^\circ}\left(1_{((\widetilde{\mathbf{T}}^\tau)^\circ )^{\tau F_0}}\right) \circ \pr =R_{c_x(\mathbf{T}_1)}^{\mathbf{G}}\left(1_{c_x(\mathbf{T}_1)^{F_0}}\right) .$$
By \cite[Proposition 4.8]{dignemichel1994grpsnonconn} we have thereby 
\begin{align*}
\langle \hat{\underline{\eta}},\hat{\underline{\eta}}\rangle_{(\underline{\mathbf{G}} \rtimes \langle \tau \rangle)^{\tau F_0}}&= \left\langle  R_{(\widetilde{\mathbf{T}}^\tau)^\circ}^{((\underline{\mathbf{G}} \rtimes \langle \tau \rangle)^{\tau})^\circ}\left(1_{((\widetilde{\mathbf{T}}^\tau)^\circ)^{\tau F_0}} \right),R_{(\widetilde{\mathbf{T}}^\tau)^\circ}^{((\underline{\mathbf{G}} \rtimes \langle \tau \rangle)^{\tau})^\circ}\left(1_{((\widetilde{\mathbf{T}}^\tau)^\circ)^{\tau F_0}} \right) \right\rangle_{((\underline{\mathbf{G}} \rtimes \langle \tau \rangle)^{\tau})^\circ)^ {\tau F_0 }}\\
&=\left\langle R_{c_x(\mathbf{T}_1)}^{\mathbf{G}}\left(1_{c_x(\mathbf{T}_1)^{F_0}}\right), R_{c_x(\mathbf{T}_1)}^{\mathbf{G}}\left(1_{c_x(\mathbf{T}_1)^{F_0}}\right) \right\rangle_{\mathbf{G}^{ F_0}}= \langle \underline{\eta},\underline{\eta}\rangle_{\underline{\mathbf{G}} ^{\tau F_0}}.
\end{align*}
The last equality holds since the multiplicities of the constituents of the Deligne--Lusztig characters of $\mathbf{G}^{F_0}$ and $\mathbf{G}^{F_0^k}$ coincide. 
Since all constituents of $\underline{\eta}$ have multiplicity $\pm 1$, the scalar product with itself gives us the number of irreducible constituents of $\underline{\eta}$. Now all constituents of $\underline{\eta}$ extend to $\underline{\mathbf{G}}^{\tau F_0} \rtimes \langle \tau \rangle$ and we have the same number of irreducible constituents in $\hat{\underline{\eta}}$. Thus, every constituent of $\underline{\eta}$ has a unique extension that is a constituent of $\hat{\underline{\eta}}$. For $\underline{\psi}$, we denote this character by $\hat{\underline{\psi}}$.

The Deligne--Lusztig character $\eta$ is $\mathcal{H}$-invariant because it arises from the trivial character. Analogously, $\underline{\eta}$ and $\underline{\hat{\eta}}$ are also $\mathcal{H}$-invariant and from the uniqueness we know that $\hat{\underline{\psi}}$ is $ \mathcal{H}_\psi$-invariant. It follows that $$\left(c_h \circ({\pr}_{\mathbf{G}\rtimes \langle \tau \rangle }^{\vee})^{-1}\right)(\hat{\underline{\psi}}) \in \Irr\left(\mathbf{G}^{F} \rtimes \langle c_h \circ F_0 \circ c_{h^{-1}}  \rangle\right)$$ is an $\mathcal{H}_\psi$-invariant extension of $\psi$. Extending this character canonically to all inner automorphisms yields an $ \mathcal{H}_\psi$-invariant extension of $\psi$ to $\mathbf{G}^F \rtimes   \Aut(\mathbf{G}^F)$ that we can restrict to $\mathbf{G}^F \rtimes \Gamma$.
\end{proof}
\begin{prop} \label{propregularext}
For every regular character $\psi$ of $\mathbf{G}^F$ there exists a $(\Gamma \times \mathcal{H})_\psi$-invariant extension $\hat{\psi} \in \Irr(\mathbf{G}^F \rtimes \Gamma_\psi)$.
\end{prop}
\begin{proof}
We write $D:=\langle \gamma \rangle_\psi$. Let $\mathbf{U}$ be the unipotent radical of an $F$-stable Borel subgroup of $\mathbf{G}$ and $\xi$ be a linear and $\langle \gamma \rangle$-invariant character of $\mathbf{U}^F$ inducing the Gelfand--Graev character $\Theta_1$ as in \cite[Section 5.2]{ruhstorfer2017navarro}. 
We extend it trivially to a character $\hat{\xi} \in \Irr( \mathbf{U}^F \rtimes D)$. As in Remark \ref{remextension}(\ref{reminduction}), $\hat{\Theta}_1:=\Ind_{ \mathbf{U}^F \rtimes D}^ { \mathbf{G}^F \rtimes D}(\hat{\xi})$ is an extension of $\Theta_1$. 
Since $\psi$ is regular, it is a constituent of $\Theta_1$ with multiplicity $1$ \cite[Corollary 12.4.10]{dignemichel2020representations}. Thus, there exists a unique constituent of $\hat{\Theta}_1$ extending $\psi$ to $\mathbf{G}^F \rtimes D$. We denote this character by $\hat{\psi}$. 

Let $e \in \mathbb{N}$ with $(\gamma^e, \sigma) \in (\Gamma \times \mathcal{H})_\psi$. We know by \cite[Lemma 3.5]{johansson2020} that there exists a $t \in \mathbf{T}^\gamma$ such that $\xi^\sigma=\xi^t$. Thus, $\gamma^e \sigma^{-1} t$ fixes $\xi$, $\hat{\xi}$, $\Theta_1$, and $\hat{\Theta}_1$. Since ${\Theta}_1$ is multiplicity free \cite[Theorem 12.3.4]{dignemichel2020representations}, $\hat{\Theta}_1$ is also multiplicity free and $t \in \mathbf{T}^\gamma \subset \mathbf{G}^F$ acts trivially on $\hat{\psi}$. It follows that $\hat{\psi}$ is also $\gamma^e \sigma^{-1}$-invariant. 
Extending it to the inner automorphisms in $\Gamma_\psi$ completes the proof.
\end{proof}
\subsection{Local characters}
To prove an analogous result for the local characters, we look at the characters of $N=N_{\mathbf{G}^F}(\mathbf{S})$ for Sylow $\Phi^{(p)}$-tori $\mathbf{S}$ as in Section \ref{sectionbijection}. We continue to use the notation introduced there. As in the proof of Proposition \ref{propequivbijection}, we know that the irreducible $p'$-characters of $N$ can be parametrized by $(s, \eta) \in \mathcal{M}$ via 
$\psi^{(N)}(s, \eta)= \Ind_{N_{\chi_s^{{\mathbf{T}}}}}^N (\Lambda(\chi_s^{{\mathbf{T}}})(\eta \circ i_{s,1})) .$ 
We use this to consider the irreducible characters of $N$ case by case.
\begin{prop}\label{propextlocal}
Let $\psi$ be an irreducible $p'$-character of $N$. Then there exists a $(\Gamma \times \mathcal{H})_\psi$-invariant extension $\hat{\psi} \in \Irr(N \rtimes \Gamma_\psi)$.
\end{prop}

\begin{proof}
Let $D \leq \Gamma_\psi$ be a cyclic subgroup of $\Gamma$ containing representatives for all outer automorphisms in $\Gamma_\psi$ and let $\gamma'$ be a generator of $D$. We can assume that $D$ has odd order since we can otherwise consider $\langle \gamma'^{|D|_{2}} \rangle$.
An overview of the local characters is given in Table \ref{table2b2local}, \ref{table2g2local}, and \ref{table2f4local}.

We first consider $(s,\eta) \in \mathcal{M}$. If $s=1$, the corresponding character $\psi:=\psi^{(N)}(1, \eta)$ is the character inflated from $\eta \circ i_{1,1}$ where $\eta$ is a character of the corresponding relative Weyl group $\mathbf{W}_{\mathbf{G}^F}(\mathbf{T})$.
The relative Weyl group is, depending on $p$ and the type of $\mathbf{G}$, either cyclic or isomorphic to $D_{16}$, $\GL_2(3)$, or the complex reflection group with Sheppard-Todd number $8$. We see that every character of these groups is linear, rational, or a constituent of $\Ind_H^{\mathbf{W}_{\mathbf{G}^F}(\mathbf{T})}(\tau)$ of multiplicity $1$ for some subgroup $H \leq \mathbf{W}_{\mathbf{G}^F}(\mathbf{T})$ and a $D$-invariant linear character $\tau \in \Irr(H)$. Thus, we find a $(\Gamma \times \mathcal{H})_\psi$-invariant extension of $\eta$ to $\mathbf{W}_{\mathbf{G}^F}(\mathbf{T}) \rtimes D$ by Remark \ref{remextension}(\ref{reminduction}) and Lemma \ref{lemmaoddorderreal}. Inflation yields a $(\Gamma \times \mathcal{H})_\psi$-invariant extension of $\psi$ to $N \rtimes D$ and we can extend this character to the remaining inner automorphisms in $\Gamma$.

\begin{table}[htbp]
\renewcommand{\arraystretch}{1.5}
\begin{tabular}{|l|l|l|}
\multicolumn{3}{l}{
$p \mid q^2-1$, $\mathbf{T}^F \cong C_{q^2-1} $, $N=\mathbf{T}^F \rtimes C_2$} \\
\hline
$s$	&$|s^N|$ &  Associated characters $\psi^{(N)}(s, \cdot)$ \\ \hline
$1$	& $1$	& $\Inf_{C_2}^N(\eta)$ for $\eta \in \Irr(C_2)$ \\ \hline
$\neq 1$ &$2$	& $\Ind_{\mathbf{T}^F}^N(\chi_s^{\mathbf{T}} )$  \\ \hline
\end{tabular}
\begin{tabular}{|l|l|l|}
\multicolumn{3}{l}{$p \mid q^2\pm \sqrt{2}q+1$, $\mathbf{T}^F \cong C_{q^2\pm \sqrt{2}q+1} $, $N=\mathbf{T}^F \rtimes C_4$} \\
\hline
$s$	&$|s^N|$ &  Associated characters $\psi^{(N)}(s, \cdot)$ \\ \hline
$1$	& $1$	& $\Inf_{C_4}^N(\eta)$ for $\eta \in \Irr(C_4)$ \\ \hline
$\neq 1$ &$4$	& $\Ind_{\mathbf{T}^F}^N(\chi_s^{\mathbf{T}} )$  \\ \hline 
\end{tabular}
\caption{Description of Sylow $\Phi^{(p)}$-tori for type $^2\mathsf{B}_2$, their normalizers $N$, the size of their $N$-conjugacy classes and the associated irreducible characters \cite[Section 16]{isaacsmallenavarro2007}.}
\label{table2b2local}
\end{table}
\begin{table}[htbp]
\renewcommand{\arraystretch}{1.5}
\begin{minipage}[t]{0.5\textwidth}
\begin{tabular}[t]{|l|l|l|}
\multicolumn{3}{l}{
$p \mid q^2-1$, $\mathbf{T}^F \cong C_{q^2-1}$, $N=\mathbf{T}^F \rtimes C_2$} \\ \hline
$s$	&$|s^N|$ &  Associated characters $\psi^{(N)}(s, \cdot)$ \\ \hline
$1$	& $1$	& $\Inf_{C_2}^N(\eta)$ for $\eta \in \Irr(C_2)$ \\ \hline
$\neq 1$ &$2$	& $\Ind_{\mathbf{T}^F}^N(\chi_s^{\mathbf{T}} )$  \\ \hline
\multicolumn{3}{l}{$p \mid q^2\pm \sqrt{3}q+1$, $\mathbf{T}^F \cong C_{q^2\pm \sqrt{3}q+1} $, $N=\mathbf{T}^F \rtimes C_6$} \\
\hline
$1$	& $1$	& $\Inf_{C_6}^N(\eta)$ for $\eta \in \Irr(C_6)$ \\ \hline
$\neq 1$ &$6$	& $\Ind_{\mathbf{T}^F}^N(\chi_s^{\mathbf{T}} )$  \\ \hline 
\end{tabular}
\end{minipage}
\begin{minipage}[t]{0.46\textwidth}
\begin{tabular}[t]{|p{0.22\textwidth}|p{0.08\textwidth}|p{0.55\textwidth}|}
\multicolumn{3}{p{\textwidth}}{
$p \mid q^2+1,$ $\mathbf{T}^F \cong C_{(q^2+1)/2} \times C_2,$  $N=\mathbf{T}^F \rtimes C_6$
} \\
\hline
$s$	&$|s^N|$ &  Associated characters $\psi^{(N)}(s, \cdot)$ \\ \hline
$1$	& $1$	& $\Inf_{C_6}^N(\eta)$ for $\eta \in \Irr(C_6)$ \\ \hline
$\ord(s)=2$ &$3$	& $\Ind_{\mathbf{T}^F \rtimes C_2}^N(\Lambda(\chi_s^{\mathbf{T}})(\eta \circ i_{s,1}) )$, $\eta \in \Irr(C_2)$ \\ \hline 
$\ord(s)>2$ &$6$& $\Ind_{\mathbf{T}^F}^N(\chi_s^{\mathbf{T}} )$  \\ \hline 
\end{tabular}
\end{minipage}
\caption{Description of Sylow $\Phi^{(p)}$-tori for type $^2\mathsf{G}_2$, their normalizers $N$, the size of their $N$-conjugacy classes and the associated irreducible characters \cite[Section 17]{isaacsmallenavarro2007}.}
\label{table2g2local}
\end{table}
\begin{table}[htbp]
\renewcommand{\arraystretch}{1.5}
\begin{tabular}{|p{0.16\textwidth}|p{0.09\textwidth}|p{0.04\textwidth}|p{0.6\textwidth}|}
\multicolumn{4}{p{\textwidth}}{
$p \mid q^2-1$, $\mathbf{T}^F \cong C_{q^2-1} \times C_{q^2-1} \cong \langle \alpha \rangle \times  \langle \alpha '\rangle$, $N=\mathbf{T}^F \rtimes D_{16}$} \\
\hline
$s$& Type in \cite{shinoda1975conj2f4}	& $|s^N|$ &  Associated characters $\psi^{(N)}(s, \cdot)$ \\ \hline
$(1,1)$			&$t_0$	& $1$	& $\Inf_{D_{16}}^N(\eta)$ for $\eta \in \Irr(D_{16})$ \\ \hline
$(1,\alpha'^i)$ &$t_1$	& $8$	& $\Ind_{\mathbf{T}^F \rtimes C_2}^N(\Lambda(\chi_s^{\mathbf{T}}) (\eta \circ i_{s,1}))$ for $\eta \in \Irr(C_{2})$  \\ \hline
$(\alpha^i,\alpha'^{i(\sqrt{2}q+1)})$ &$t_2$	& $8$	& $\Ind_{\mathbf{T}^F \rtimes C_2}^N(\Lambda(\chi_s^{\mathbf{T}}) (\eta \circ i_{s,1}))$ for $\eta \in \Irr(C_{2})$  \\ \hline
$(\alpha^i,\alpha'^j)$ &$t_3$	& $16$	& $\Ind_{\mathbf{T}^F }^N(\chi_s^{\mathbf{T}})$  \\ \hline 
\multicolumn{4}{p{\textwidth}}{$p \mid q^2+1$, $\mathbf{T}^F \cong C_{q^2+1} \times C_{q^2+1} \cong \langle \alpha \rangle \times  \langle \alpha '\rangle$, $N=\mathbf{T}^F \rtimes \GL_2(3)$} \\
\hline
$(1,1)$			&$t_0$	& $1$	& $\Inf_{\GL_2(3)}^N(\eta)$ for $\eta \in \Irr(\GL_2(3))$ \\ \hline
$(\alpha^{i\frac{q^2+1}{3}},$ $\alpha'^{j\frac{q^2+1}{3}})$ &$t_4$	& $8$	& $\Ind_{\mathbf{T}^F \rtimes H}^N(\Lambda(\chi_s^{\mathbf{T}}) (\eta \circ i_{s,1}))$ for $H \leq \GL_2(3)$ of size $6$, $\eta \in \Irr(H)$ \\ \hline
$(\alpha^i,\alpha'^i)$ &$t_5$	& $24$	& $\Ind_{\mathbf{T}^F \rtimes C_2}^N(\Lambda(\chi_s^{\mathbf{T}}) (\eta \circ i_{s,1}))$ for $\eta \in \Irr(C_{2})$ \\ \hline
$(\alpha^i,\alpha'^j)$ &$t_{14}$	& $48$	& $\Ind_{\mathbf{T}^F }^N(\chi_s^{\mathbf{T}})$ \\ \hline
\multicolumn{4}{p{\textwidth}}{
$p \mid q^2\pm \sqrt{2}q+1$, $\mathbf{T}^F \cong C_{q^2\pm \sqrt{2}q+1} \times C_{q^2\pm \sqrt{2}q+1} \cong \langle \alpha \rangle \times  \langle \alpha' \rangle$, $N=\mathbf{T}^F \rtimes U$ with $U$ the complex reflection group with Shepard--Todd number $8$} \\
\hline
$(1,1)$			&$t_0$	& $1$	& $\Inf_{U}^N(\eta)$ for $\eta \in \Irr(U)$ \\ \hline
$(1,\alpha'^i)$ &$t_7$ or $t_9$	& $24$	& $\Ind_{\mathbf{T}^F \rtimes H}^N(\Lambda(\chi_s^{\mathbf{T}}) (\eta \circ i_{s,1}))$ for $H\leq U$ of size $4$, $\eta \in \Irr(H)$  \\ \hline
$(\alpha^i,\alpha'^j)$ &$t_{12}$ or $t_{13}$	& $96$	& $\Ind_{\mathbf{T}^F }^N(\chi_s^{\mathbf{T}})$  \\ \hline
\multicolumn{4}{p{\textwidth}}{
$p \mid q^4 -q^2+1$, $\mathbf{T}^F \cong C_{q^4 -q^2+1} \cong \langle \alpha \rangle $, $N=\mathbf{T}^F \rtimes C_6$ } \\
\hline
$1$			&$t_0$	& $1$	& $\Inf_{C_6}^N(\eta)$ for $\eta \in \Irr(C_6)$ \\ \hline
$\alpha^{\frac{q^4 -q^2+1}{3}}$ &$t_4$& $2$	& $\Ind_{\mathbf{T}^F \rtimes C_3}^N(\Lambda(\chi_s^{\mathbf{T}}) (\eta \circ i_{s,1}))$ for $\eta \in \Irr(C_3)$  \\ \hline
$\alpha^i$ &$t_{15}$& $6$	& $\Ind_{\mathbf{T}^F }^N(\chi_s^{\mathbf{T}})$ \\ \hline 
\multicolumn{4}{p{\textwidth}}{$p \mid q^4 \pm \sqrt{2}q^3 + q^2\pm \sqrt{2}q+1$, $\mathbf{T}^F \cong C_{q^4 \pm \sqrt{2}q^3 + q^2\pm \sqrt{2}q+1} \cong \langle \alpha \rangle $, $N=\mathbf{T}^F \rtimes C_{12}$ } \\ \hline
$1$			&$t_0$	& $1$	& $\Inf_{C_{12}}^N(\eta)$ for $\eta \in \Irr(C_{12})$ \\ \hline
$\alpha^i$ &$t_{16}$ or $t_{17}$& $12$	& $\Ind_{\mathbf{T}^F }^N(\chi_s^{\mathbf{T}})$   \\ \hline
\end{tabular}
\caption{Description of Sylow $\Phi^{(p)}$-tori for type $^2\mathsf{F}_4$, their normalizers $N$ given in \cite{malle1991maximalsubgroups}, representatives of its semisimple $N$-conjugacy classes from \cite{shinoda1975conj2f4}, and the associated irreducible characters. Here, $\alpha$ and $\alpha'$ denote generators of the respective cyclic groups.}
\label{table2f4local}
\end{table}

We assume that $s$ is not trivial and consider $\psi:=\psi^{(N)}(s, \eta)=\Ind_{N_{\chi_s^{{\mathbf{T}}}}}^N (\psi_0)$ for $\psi_0:=\Lambda(\chi_s^{{\mathbf{T}}})(\eta \circ i_{s,1})$. 
Since $\mathbf{T}^F$ is abelian and $\Lambda$ preserves the character degrees, $\Lambda(\chi_s^{{\mathbf{T}}})$ is always linear. First assume that $\eta$ is linear. 
Because $\gamma'$ fixes $\psi$ there exists an $n \in N$ such that $\psi_0^{\gamma'n}=\psi_0$. Thus, we find a $(\langle \gamma' n\rangle \times \mathcal{H})_\psi$-invariant extension $\hat{\psi}$ of $\psi$ to $N \rtimes \langle \gamma' n \rangle$ by Remark \ref{remextension}(\ref{remlinear}) and (\ref{reminduction}). Since $\Inn(N)$ acts trivially on $\psi$, the extension $\hat{\psi}$ is also $(D \times \mathcal{H})_\psi$-invariant.  The claim follows by extending $\hat{\psi}$ to the remaining inner automorphisms of $N$.

For type $^ 2\mathsf{F}_4$ it can happen that $\eta \in \Irr(H)$ is not linear but integer-valued. Let $n \in N$ such that $(\chi_s^T)^{\gamma'n}=\chi_s^T$ and $\gamma' n$ has odd order. We have a natural isomorphism $(\mathbf{T}^F \rtimes H)\rtimes \langle \gamma' n \rangle \cong \mathbf{T}^F \rtimes (H\rtimes \langle \gamma' n \rangle)$.  By \cite[Proposition 2.3]{cabanesspaeth2013}, $\eta \circ i_{s,1}$ is $\gamma' n$ -stable and with Lemma \ref{lemmaoddorderreal} we can find a $(\Gamma \times \mathcal{H})_\psi$-invariant extension of $\eta \circ i_{s,1}$ to $H \rtimes  \langle \gamma' n \rangle$. Inflating this character $\hat{\eta}$ and using the construction of irreducible characters of semidirect products from \cite[Section 8.2]{serre1977book}, we obtain an $(\Gamma \times \mathcal{H})_\psi$-invariant extension of $\Lambda(\chi_s^{\mathbf{T}}) (\eta \circ i_{s,1})$ to $\mathbf{T}^F \rtimes (H\rtimes \langle \gamma' n \rangle)$. With Remark \ref{remextension}(\ref{reminduction}) we get a $(\Gamma \times \mathcal{H})_\psi$-invariant extension of $\psi$ to $N \rtimes \Gamma_\psi$. 
\end{proof}
\subsection{Verification of the inductive condition}
\begin{theorem}
The group $\mathbf{G}^F$ as in Section 2.3 satisfies Condition \ref{imck} for $p$. 
\end{theorem}
\begin{proof}
We use the notation established in the previous sections.
We can choose $N:=N_{\mathbf{G}^F}(\mathbf{T})$ by \cite[Theorem 8.4]{malle2008height0} and we know from Proposition \ref{propequivbijection} that there exists a $\Gamma \times \mathcal{H}$-equivariant bijection $\Omega: \Irr_{p'}(N) \rightarrow  \Irr_{p'}(\mathbf{G}^F).$ Looking at the generic character table of $\mathbf{G}^F$ as given in \cite{chevie}, we see that every irreducible global $p'$-character is either unipotent, regular, real, or semisimple. 
Thus, every $\psi \in \Irr_{p'}(\mathbf{G}^F)$ can be extended to a $(\Gamma \times \mathcal{H})_\psi$-invariant character of $\mathbf{G}^F \times \Gamma_\psi$ by Proposition \ref{propextunip}, Propostion \ref{propregularext},  Lemma \ref{lemmaoddorderreal}, or \cite[Proposition 3.6]{johansson2020}. The same is possible for the local characters by Proposition \ref{propextlocal}. 
\end{proof}

\section{Groups with non-generic Sylow normalizers} \label{sectionnonconvenient}
In this section, we consider the groups and primes for which the normalizer of a Sylow $p$-subgroup does not have to lie in the normalizer of a Sylow $\Phi^{(p)}$-torus. This happens for the primes and Suzuki and Ree groups that were excluded before as well as for some groups of type $\mathsf{A}_2$, $^2\mathsf{A}_2$, and $\mathsf{G}_2$ for the prime $p=3$ \cite[Theorem 5.14 and Theorem 8.4]{malle2008height0} and of type $\mathsf{C}_n$ for the prime $p=2$ \cite[Theorem 5.19]{malle2008height0}. 
The latter case has already been settled in \cite[Theorem A]{ruhstorfer2021inductive}. 
We consider the remaining groups since we want to complete the case of Suzuki and Ree groups and it seems natural to look at the special cases for other groups of Lie type.
\subsection{Suzuki and Ree groups} We verify the inductive McKay--Navarro condition by looking at the global and local characters case by case. The automorphism group is known from above.
\begin{prop}\label{propspecialg2f4}
Condition \ref{imck} is satisfied for $f \geq 1$ and
\begin{enumerate}[(a)]
\item ${}^2\mathsf{G}_2(3^{2f+1})$ and $p=2$;
\item 
${}^2\mathsf{F}_4(2^{2f+1})$, $p=3$ and $2^{2f+1} \equiv 2,5 \mod 9$.
\end{enumerate}
\end{prop}
\begin{proof}
First let $G={}^2\mathsf{G}_2(3^{2f+1})$ and $p=2$. As in \cite[Proof of Theorem 17.1]{isaacsmallenavarro2007}, $R \in \Syl_2(G)$ is elementary abelian of order $8$ and $N=N_G(R)=R \cdot H$ with $H \cong C_7 \rtimes C_3$ a Frobenius group. 
Then, we have
\begin{align*}
{\Irr}_{3'}(N)=\{{\theta}_i={\Inf}_{C_3}^N(\widetilde{\theta}_i) \mid \widetilde{\theta_i} \in \Irr(C_3), 1 \leq i \leq 3 \} \cup \{\psi_i={\Inf}_{H}^N({\Ind}_{C_7}^H(\tau_i)) \mid 1 \neq \tau_i \in \Irr(C_7) \} &\\ \cup \{\lambda_1, \lambda_2 , \lambda_3 \mid \lambda_1+\lambda_2+\lambda_3= {\Ind}_{R}^N(\chi), \; 1 \neq \chi \in \Irr(R)\} &,
\end{align*}
i.e. $N$ has three linear characters, two characters of degree $3$, and three characters of degree $7$. From \cite[Proof of Theorem 17.1]{isaacsmallenavarro2007}, we know that $F_3$ acts trivially on $N$ and all of these characters. An explicit construction of the character values shows us, after choosing a suitable labeling of the characters, that
$$ {\theta}_2^\sigma = \left\{ \begin{array}{ll}
{\theta}_2 & \textrm{ if } \xi_3^\sigma=\xi_3, \\
{\theta}_3 & \textrm{ if } \xi_3^\sigma=\xi_3^2, \\
\end{array} \right., \quad
\lambda_2^\sigma = \left\{ \begin{array}{ll}
\lambda_2 & \textrm{ if } \xi_3^\sigma=\xi_3, \\
\lambda_3 & \textrm{ if } \xi_3^\sigma=\xi_3^2, \\
\end{array} \right.
$$
where $\xi_3 \in \mathbb{C}^\times$ is a third root of unity. All other characters of $N$ are fixed by $\mathcal{H}$. 

It is clear that the linear characters $\theta_i$ can be extended to $G \rtimes \langle F_3 \rangle$ such that they are $\mathcal{H}_{\theta_i }$-invariant. For the characters $\psi_1$ and $\psi_2,$ an extension can be found as in Remark \ref{remextension}(\ref{reminduction}) and by inflating the obtained character. 
In the same way, we see that there is an $\mathcal{H}$-invariant extension of $\lambda_1+\lambda_2+\lambda_3= {\Ind}_{R}^N(\chi)$ to $N \rtimes \langle F_3 \rangle$. This extension $\hat{\lambda}$ decomposes into three irreducible characters $\hat{\lambda}_1$, $\hat{\lambda}_2$, $\hat{\lambda}_3$ that extend $\lambda_1,\lambda_2$ and $\lambda_3$. It is easy to see that they are $\mathcal{H}_{\lambda_i}$-invariant, respectively. 

In the global case, we have  $$\Irr_{2'}(G)=\{\chi_1,\chi_4,\chi_5, \chi_6, \chi_7, \chi_8, \chi_a, \chi_a'\}$$ in the notation of \cite[Table 4]{isaacsmallenavarro2007}. The automorphism $F_3$ acts trivially on them since they are unipotent or uniquely determined by their degree. The Galois automorphism $\sigma \in  \mathcal{H}$ acts by permuting the indices as $(4,5)(6,7)$ if we have $\xi_3^\sigma=\xi_3^2$ and trivially otherwise. Thus, there exists a $\Gamma \times \mathcal{H}$-equivariant bijection. We find $\mathcal{H}_{\chi}$-invariant character extensions to $G \rtimes \langle F_3 \rangle $ for all $\chi \in \Irr_{2'}(G)$ by Lemma \ref{lemmaoddorderreal} and Proposition \ref{propextunip}. Thus, Condition \ref{imck} is satisfied. 

Let $G={}^2\mathsf{F}_4(2^{2f+1})$ for $f \geq 1$, $p=3$ and assume $2^{2f+1} \equiv 2,5 \mod 9$. 
In the global case, we see that $G$ has six unipotent characters of $3'$-degree and three other $3'$-characters that all have distinct degrees \cite{chevie}. 
Thus, the characters are all $\Gamma$-invariant, $\mathcal{H}$-invariant, and real.
Since $F_2$ has odd order, there exists a unique $\mathcal{H}$-invariant extension to $G \rtimes \Gamma $ for every character in $\Irr_{3'}(G)$ by Lemma \ref{lemmaoddorderreal}. 

As mentioned in \cite[Proof of Theorem 8.4]{malle2008height0}, the normalizer of a Sylow $3$-subgroup of $G$ is $N \cong \SU_3(2).2$. Thus, we can explicitly compute ${\Irr}_{3'}(N)=\{\chi_i \mid 1 \leq i \leq 9\}.$
We know by \cite[Proof of Proposition 3.16]{malle2008extuni} that $\Gamma$ acts trivially on ${\Irr}_{3'}(N)$. By looking at the character values, we see that $\mathcal{H}$ acts trivially on the characters and all characters except $\chi_6,\chi_7$ are real.
For the characters in ${\Irr}_{3'}(N)\setminus \{\chi_6,\chi_7\}$ we obtain $\mathcal{H}$-invariant extensions to $N \rtimes \Gamma$ again by Lemma \ref{lemmaoddorderreal}.

The characters $\chi_6$, $\chi_7$ have both degree $2$ and we can show with \textsf{GAP} \cite{GAP4} that both are induced by two linear characters $\lambda_6, \lambda_6'$ and $\lambda_7, \lambda_7'$ of $((C_3 \times C_3) : C_3) : C_8 \leq N$, respectively. Since $F_2$ fixes $\chi_6$ and $\chi_7$, both $\lambda_6$ and $\lambda_7$ are fixed or mapped to $\lambda_6'$ and $\lambda_7'$, respectively.   The orbit cannot have length $2$ because $F_2$ has odd order, thus $\lambda_6$, $\lambda_7$ are fixed by $F_2$. Since $\chi_6$ is $\mathcal{H}$-invariant, $\lambda_6$ and $\lambda_6^\sigma$ are conjugate by some element of $N$ for every $\sigma \in \mathcal{H}$. It follows as before that there is an $\mathcal{H}$-invariant extension of $\chi_6$ to $N \rtimes \langle F_2 \rangle$. The same can be done for $\chi_7$ and we have verified Condition \ref{imck}.
\end{proof} 
\subsection{Special linear and special unitary groups}
Let $p=3$ and $X=\PSL_3(q)$ with $q \equiv 4,7 \mod 9$ or $X=\PSU_3(q)$ with $q \neq 2$ and $q \equiv 2,5 \mod 9.$
The inductive McKay condition was verified for $X$ and $p$ in \cite[Section 3.1 and 3.2]{malle2008extuni}. We verify the inductive McKay--Navarro condition for these groups by recalling the considerations from there and extending them to the stronger condition. Note that the assumptions for Condition \ref{imck} are not satisfied anymore and we have to consider the original inductive McKay--Navarro condition from \cite[Definition 3.5]{navarro2019reduction}.

Note that we do not consider $\PSU_3(2)$ since it is solvable.
The group $\PSL_3(4)$ has exceptional Schur multiplier and is considered separately in Proposition \ref{propsmall}. For all other groups, the Schur multiplier of $X$ has order $3$ and we can consider $X$ itself (instead of its universal covering group) by \cite[Lemma 5.1]{johansson2020}. We follow \cite[Section 3.1 and 3.2]{malle2008extuni} and use the notation from there. Let $S:=\SL_3(q)$ or $\SU_3(q)$, respectively.
Since the irreducible characters of $S$ with $Z(S)$ in their kernel are in bijection with the irreducible characters of $X \cong S/Z(S)$, we can consider the irreducible $3'$-characters of $S$ with $Z(S)$ in their kernel. We use this bijection without further notice.
\begin{lemma} \label{lemmaSLSUHinv}
The irreducible $3'$-characters $\rho_1,$ $\rho_2,$ $\rho_3$, $ \varphi_1,$ $ \varphi_2,$ $ \varphi_3$ of $S$ are invariant under $\mathcal{H}$.
\end{lemma}
\begin{proof}
The unipotent characters $\rho_1,$ $\rho_2,$ $\rho_3$ have distinct degrees and are thus $\mathcal{H}$-invariant. The character values of $ \varphi_1,$ $ \varphi_2,$ $ \varphi_3$ that are different for the three characters are given in \cite[Section 3.1 and 3.2]{malle2008extuni} and we see that they are rational.
\end{proof}
From \cite[Lemma 3.4 and Section 3.2]{malle2008extuni} we know about the structure of $\Out(S)$, the outer automorphism group of $S$. We write $D:= \langle \psi, \gamma,  \delta \rangle \cong \Out(S) \cong S_3 \times C_f$. Let $l$ be the prime with $q=l^f$ and note that the condition on $q$ ensures that $f$ is odd in the case of unitary groups. 
\begin{lemma} \label{lemmaSLHextglob}
For $X=\PSL_3(q)$, every $\tau \in \Irr_{3'}(X) $ has an $\mathcal{H}_\tau$-invariant extension to $X \rtimes D_\tau$.
\end{lemma}
\begin{proof}
Let $G:=\PGL_3(q)$ and $B$ be a Borel subgroup of $G$. The unipotent characters $\rho_1, \rho_2, \rho_3$ can be extended to $\tilde{\rho}_1, \tilde{\rho}_2, \tilde{\rho}_3 \in \Irr(G \rtimes \langle \psi \rangle)$ as in the proof of \cite[Theorem 2.4]{malle2008extuni} and we have $$\Ind^{G \rtimes \langle \psi \rangle}_{B \rtimes \langle \psi \rangle}(1)=\tilde{\rho}_1+ 2\tilde{\rho}_2+\tilde{\rho}_3.$$
Similary, we get $$\Ind^{G \rtimes \langle \gamma \rangle}_{B \rtimes \langle \gamma \rangle}(1)=\hat{\rho}_1+ 2\hat{\rho}_2+\hat{\rho}_3$$ with $\hat{\rho}_1, \hat{\rho}_2, \hat{\rho}_3 \in \Irr(G \rtimes \langle \gamma \rangle)$ extending $\rho_1, \rho_2, \rho_3$ by  \cite[p.447f]{malledarstmethoden1991}.
Since we know that all unipotent characters extend to $G \rtimes  \langle \gamma, \psi \rangle $, at least one of these extensions is a constituent of $\Ind^{G \rtimes \langle \gamma, \psi \rangle}_{B \rtimes \langle \gamma, \psi \rangle}(1)$. These constituents also extend $\hat{\rho}_i$ and $\tilde{\rho}_i$. Thus, the constituent extending $\rho_i$ is unique in the permutation character and it follows that it is $\mathcal{H}$-invariant.

The semisimple character $\varphi_1$ is a constituent of the dual of the Gelfand-Graev character $\Gamma_1$ of $S$ and satisfies $((S \rtimes \langle \delta \rangle)\langle \psi, \gamma\rangle)_{\varphi_1}=(S \rtimes \langle \delta \rangle)_{\varphi_1} \langle \psi, \gamma \rangle_{\varphi_1}=S  \langle \psi, \gamma \rangle$. As in \cite[Corollary 6.9]{ruhstorfer2017navarro} we find an extension $\hat{\varphi}_1 \in \Irr(S  \langle \psi, \gamma \rangle)$ that is $ \mathcal{H}$-invariant. Note that we can use this result since Ruhstorfer does not use any properties coming from the prime corresponding to $\mathcal{H}$. 

The characters $\hat{\varphi}_1^\delta \in \Irr(S  \langle \psi \delta, \gamma \delta \rangle)$ and $\hat{\varphi}_1^{\delta^2} \in \Irr(S  \langle \psi \delta^2, \gamma \delta^2 \rangle)$ are extensions of $\varphi_2$ and $\varphi_3$, respectively. Since they are $\mathcal{H}$-invariant, this shows the claim.

\end{proof}
\begin{lemma} \label{lemmaSUHextglob}
For $X=\PSU_3(q)$, every $\tau \in \Irr_{3'}(X) $ has an $\mathcal{H}$-invariant extension to $X \rtimes D_\tau$.
\end{lemma}
\begin{proof}
For the semisimple characters as well as for $\rho_1$ and $\rho_3$, this can be shown in the same way as for $\PSL_3(q)$. However, the unipotent character $\rho_2$ does not lie in the principal series of $X$ and has to be treated differently.

We first show that there is an $\mathcal{H}$-invariant extension of $\rho_2$ to $X \rtimes \langle \gamma \rangle $.
Assume that $l$ is odd. The two extensions of $\rho_2$ to $X \rtimes \langle \gamma \rangle $ have already been considered in \cite[Proposition 2.1]{malle1990unitarygalois} for $l \equiv 3 \mod 4$. In the same way as in the proof given there, one can show for all odd $l$ that the only non-rational character value of an extension $\hat{\rho}_2 \in \Irr(X \rtimes \langle \gamma \rangle)$ is $\sqrt{-q}$. By the law of quadratic reciprocity, $3$ is a square modulo $l$ if $l \equiv 1 \mod 4$ and not a square otherwise. Since we have
$$i\sqrt{q}=\left\{\begin{array}{ll}
l^{(f-1)/2}\sum_{k=0}^{l-1}\xi_l^{k^2} \cdot i & \text{if } l \equiv 1 \mod 4, \\
l^{(f-1)/2}\sum_{k=0}^{l-1}\xi_l^{k^2} & \text{if } l \equiv 3 \mod 4, \\
\end{array}\right.$$
with $\xi_p$ a primitive $p$-th root of unity and $i$ the imaginary unit, it is easy to see that $\sigma \in \mathcal{H}$ fixes $i\sqrt{q}$. Thus, $\hat{\rho}_2$ is $\mathcal{H}$-invariant.

If $q$ is even, we first have to determine the character values for the extensions of $\rho_2$ to $\tilde{X}:=X \rtimes \langle \gamma \rangle $. We first determine the outer conjugacy classes of $\tilde{X}$, i.e. the conjugacy classes that do not lie in $X \unlhd \tilde{X}$. We can parametrize the outer conjugacy classes as described in \cite[Method 3.1]{brunat2007extensiong2graph} by
$$\left\{(g,1)x_i \left|
\begin{array}{ll}
g \in G \text{ representative of a $\gamma$-stable conjugacy class of odd order},\\
x_i \text{ a class representative of an outer class of $2'$-elements in } C_{\tilde{X}}(g,1) \end{array}\right. \right\}.$$
With the notation of \cite{simpsonframe1973slsucharacters}, the $\gamma$-stable conjugacy classes of $X$ are
$C_1$, $C_2$, $C_3^{(0)}$, $C_5',$ $C_6^{(k,l,m)},$ and $C_7^{(rk)}.$
For an element $x \in X$ in a $\gamma$-stable conjugacy class, we have $2|C_{X}(x)|=|C_{\tilde{X}}(x,1)|$.
Thus, the centralizer orders in $\tilde{X}$ for the elements of odd order in $C_5', C_6^{(k,l,m)}, C_7^{(rk)}$ can be deduced from \cite[Table 2]{simpsonframe1973slsucharacters}.
Since the Sylow $2$-subgroups of the centralizers have order $2$, the $2$-elements of the centralizers are all conjugate. Thus, the $\gamma$-stable classes $C_5', C_6^{(k,l,m)}, C_7^{(rk)}$ correspond to one outer conjugacy class with representative $c_5'(1,\gamma), c_6^{(k,l,m)}(1,\gamma), c_7^{(rk)}(1,\gamma)$ each where we denote a representative of $C_5'$ by $c_5'$ and so on.

The number of outer conjugacy classes of $\tilde{X}$ is the number of $\gamma$-stable conjugacy classes of $X$. Therefore, there are $3$ conjugacy classes corresponding to the trivial class $C_1$. They consist of $2$-elements and will be denoted by $(1,\gamma), A, B$. 

Let $\hat{\rho}_2$ be an extension of $\rho_2$ to $\tilde{X}$. Then, the other extension of $\rho_2$ to $\tilde{X}$ is given by $\sign  \cdot \hat{\rho}_2$ where $\sign$ denotes the nontrivial extension of the trivial character of $X$ to $\tilde{X}$. The elements $c_6^{(k,l,m)}(1,\gamma)$ and $ c_7^{(rk)}(1,\gamma)$ are algebraically conjugate to an odd number of elements in classes of the same type with different parameters. Since we only have two extensions of $\rho_2$, these vanish on all $c_6^{(k,l,m)}(1,\gamma), c_7^{(rk)}(1,\gamma)$ and obviously also on $c_5'(1,\gamma)$.

Now, $(1, \gamma)$ is a rational class and since the extensions of $\rho_2$ have to be different on at least one class, $A$ and $B$ cannot be rational and it follows $B=A^{-1}$.

We now determine $|C_{\tilde{X}}(a)|$ for $a \in A$. From \cite[Lemma 3.1 2.]{brunat2007extensiong2graph} we know $C_{\tilde{X}}(c_i (1,\gamma))=C_{\tilde{X}}(c_i ) \cap C_{\tilde{X}}( (1,\gamma)).$
We have $ C_{\tilde{X}}( (1,\gamma)) \cong \PSL_2(q) \times \langle \gamma \rangle$. Since $a \in A$ has order $4$ or $8$, we have $a^2 \in C_2$ or $a^4 \in C_2$ and it follows with  \cite[Table 2]{simpsonframe1973slsucharacters} that $|C_{\tilde{X}}(a)|$ divides $2|C_{{X}}(c_2)|=2q^2\frac{q+1}{3}$. We know that half of the elements of $\tilde{X}$ lie in an outer class, i.e. 
\begin{align*}
|X|&=\frac{\tilde{X}}{|C_{\tilde{X}}( (1,\gamma))|}+ \frac{\tilde{X}}{|C_{\tilde{X}}( c_5'(1,\gamma))|} + \sum_{k,l,m}\frac{\tilde{X}}{|C_{\tilde{X}}(c_6^{(k,l,m)} (1,\gamma))|}+ \sum_{k}\frac{\tilde{X}}{|C_{\tilde{X}}( c_7^{(rk)}(1,\gamma))|}+ 2\cdot \frac{\tilde{X}}{|C_{\tilde{X}}( a)|} \\
&=\frac{1}{3} q^2(q+1)^2(q-1)(q^2-q+1).
\end{align*}
Looking at the centralizer orders in \cite[Table 2]{simpsonframe1973slsucharacters} we see that $q$ divides all summands except possibly the last one. It follows that $|C_{\tilde{X}}(a)|=4q d$ with $d$ odd. We can now use orthogonality relations and the knowledge that it is a sum of $8$-th roots of unity to determine $\hat{\rho}_2(a)$. It follows that the elements of $A$ have order $8$, $|C_{\tilde{X}}(a)|=4q$ and $$\hat{\rho}_2(a)=\sqrt{-q}=\sqrt{q/2}(\xi_8+\xi_8^3).$$
Since $q$ is an odd power of $2$, $\sqrt{q/2}$ is an integer and we easily see that $\hat{\rho}_2(a)$ is $\mathcal{H}$-invariant. Thus, $\hat{\rho}_2$ is an $\mathcal{H}$-invariant extension from $\rho_2$ to $X \rtimes \langle \gamma \rangle $.

By Lemma \ref{lemmaoddorderreal} and the results of Lusztig as stated in \cite[Proposition 2.1]{malle2008extuni}, we find a unique $\mathcal{H}$-invariant extension of $\rho_2$ to $\tilde{\rho}_2 \in \Irr(X \rtimes \langle \delta, \psi^2 \rangle)$. We know from \cite[Lemma 3.11]{malle2008extuni} that $\rho_2$ extends to all outer automorphism. Thus, we find an extension $\accentset{\approx}{\rho}_2$ extending both $\tilde{\rho}_2$ and $\hat{\rho}_2$. 
Since ${\accentset{\approx}{\rho}_2}^ \sigma = \beta \accentset{\approx}{\rho}_2$ for some $\beta \in \Irr(D)$ implies $\beta|_{\langle \delta, \psi^2 \rangle}=1$ and $\beta|_{\langle \gamma \rangle}=1$, we have found an $\mathcal{H}$-invariant extension of $\rho_2$ to $X \rtimes D$.
\end{proof}
Now we turn to the local characters. As in the proof of \cite[Theorem 3.12]{malle2008extuni}, there is a natural embedding $\SU_3(q) \rightarrow \SL_3(q^2)$ that maps the normalizers of Sylow $3$-subgroups $N \cong 3^{1+2}.Q_8$ onto another where $Q_8$ is the quaternion group. Since every automorphism of $\SU_3(q)$ naturally extends to $\SL_3(q^2)$, it suffices to study the characters and character extensions for $N$ as a subgroup of $\SL_3(q)$. 
We know from \cite[Table 2]{malle2008extuni} that all six $3'$-characters $\rho_1',$ $\rho_2',$ $\rho_3'$, $ \varphi_1',$ $ \varphi_2',$ $ \varphi_3'$ of $N$ are rational and $N$ is $D$-stable. Thus, we have $\Gamma \cong \langle \Inn(N), \psi, \gamma,  \delta \rangle_\tau \cong \hat{H} \times \langle \psi \rangle$ in the inductive condition.
\begin{lemma} \label{lemmaSLSUecxtlocal}
Every $\tau \in \Irr_{3'}(N) $ has an $\mathcal{H}_\tau$-invariant extension to $N\rtimes D_\tau$.
\end{lemma}
\begin{proof}
We only have to consider the non-linear characters. We already know from \cite[Proof of Lemma 3.7]{malle2008extuni} that $\rho_2'$ and $\rho_3'$ both extend to $\hat{H} \times \langle \psi \rangle$ since we can compute the irreducible characters of $\hat{H}$ explicitly and extend them trivially to the direct product. Further, we see in \cite[Table 2]{malle2008extuni} that the extension of $\rho_3'$ is rational and looking at the character values of the extensions of $\rho_2'$ shows that they are non-rational but $\mathcal{H}$-invariant. This shows the claim.
\end{proof}

\begin{theorem}
The  groups 
 $\PSL_3(q) $ for $q \equiv 4,7 \mod 9$   
and $\PSU_3(q) $ for $q \neq 2$, $ q \equiv 2,5 \mod 9$
satisfy the inductive McKay--Navarro condition for $p=3$.
\end{theorem}
\begin{proof}
The group $\PSL_3(4)$ is considered in Proposition \ref{propsmall}.
As described before, $\Irr_{3'}(\PSL_3(q))$ is in bijection with $\Irr_{3'}(\SL_3(q))$ via inflation. We verify the condition for the local subgroup $N$. Since $\mathcal{H}$ acts trivially on the local and global characters by \ref{lemmaSLSUHinv}, the bijection from \cite[Proposition 3.8]{malle2008extuni} is also $\mathcal{H}$-equivariant and for all $\tau \in \Irr_{3'}(\SL_3(q))$ we have $(\Gamma \times \mathcal{H})_\tau=\Gamma_\tau \times \mathcal{H}$. Thus, it suffices to extend the $\mathcal{H}$-invariant character extensions in $\Irr(G\rtimes D_\tau)$ and $\Irr(N\rtimes D_\tau)$ from Lemma \ref{lemmaSLHextglob} and \ref{lemmaSLSUecxtlocal} to the remaining inner automorphisms in $\Gamma$. The resulting $\mathcal{H}$-invariant characters all have $C_{G \rtimes \Gamma}(N)$ in their kernel. Thus, the inductive McKay--Navarro condition in \cite[Definition 3.5]{navarro2019reduction} is satisfied.
We can argue analogously for $ \PSU_3(q)$.
\end{proof}
\subsection{The groups $\mathsf{G_2(q)}$} Let $G:=\mathsf{G_2(q)}$ for $q=2,4,5,7 \mod 9$ and $p=3$. We assume $q \neq 2$ since the group $\mathsf{G}_2(2)$ is considered in Proposition \ref{propsmall}. Again, we follow \cite[Section 3.3]{malle2008extuni}. The subgroup
$$M:=\left\{ \begin{array}{ll}
\SL_3(q).2 \text{ (extension with the graph automorphism)} & \text{if } q \equiv 4,7 \mod 9,\\
\SU_3(q).2 \text{ (extension with the graph-field automorphism)} & \text{if } q \equiv 2,5 \mod 9
\end{array}\right.$$
contains the normalizer of a Sylow $3$-subgroup of $G$. The outer automorphism group of $G$ is generated by a field automorphism $\psi$ and we can assume that $M$ is $\psi$-stable. 

We  know from \cite[Section 3.3]{malle2008extuni} that there are each nine global and local $3'$-characters that are all $\psi$-invariant.
\begin{lemma}
There is a $\Gamma \times \mathcal{H}$-equivariant bijection $\Irr_{3'}(G) \rightarrow \Irr_{3'}(M)$.
\end{lemma}
\begin{proof}
Looking at the character values of the global characters in \cite{changree1974charg2} and \cite{enomotoyamada1986charg2} we see that $\mathcal{H}$ acts trivially on them. 
We already obtained the irreducible local characters of $M'$ (that is $\SL_3(q)$ or $\SU_3(q)$) in the previous section. The irreducible characters of $M$ are $\Ind_{M'}^M(\varphi_2)$ and the extensions of $\rho_1$, $\rho_2$, $\rho_3$, and $\varphi_1$ to $M$. We know that one of the extensions is $\mathcal{H}$-invariant. Thus, the other extension also has to be $\mathcal{H}$-invariant. Since $\Ind_{M'}^M(\varphi_2)$ is the only irreducible character of its degree, it is also $\mathcal{H}$-invariant.
As described in \cite[Proposition 3.14]{malle2008extuni}, $\Gamma$ acts trivially on all occurring characters.
\end{proof}

\begin{lemma}
The irreducible $3'$-characters of $M$ have $\mathcal{H}$-invariant extensions to $M \rtimes \langle \psi \rangle$.
\end{lemma}
\begin{proof}
For the extensions of $\rho_1$, $\rho_2$, $\rho_3$, and $\varphi_1$ this is clear from the previous section. We also know that there is an $\mathcal{H}$-invariant extension of $\varphi_2$ to $\SL_3(q) \rtimes \langle \psi \rangle,$ $\SL_3(q) \rtimes \langle \psi \gamma \rangle,$ and $\SU_3(q) \rtimes \langle \psi^2 \rangle$, respectively. Inducing this to $M \rtimes  \langle \psi \rangle$ leads to an $\mathcal{H}$-invariant extension of $\Ind_{M'}^M(\varphi_2)$.
\end{proof}

\begin{lemma}
The irreducible $3'$- characters of $G$ have $\mathcal{H}$-invariant extensions to $G \rtimes \langle \psi \rangle$.
\end{lemma}
\begin{proof}
First assume $q$ odd. Then, all unipotent $3'$-characters are in the principal series of $G$ by \cite[p.402]{changree1974charg2}. They can be extended to $G \rtimes \langle \psi \rangle$ as described before in the proof of Lemma \ref{lemmaSLHextglob}.
In the notation of \cite{changree1974charg2}, the remaining three irreducible $3'$-characters are denoted by $X_{31},X_{32},X_{33}$. We know $R_{\mathbf{T}}^{\mathbf{G}}(\theta)=X_{31}+X_{32}-X_{33}$ for some $\theta \in \Irr(\mathbf{T}^F)$ and a torus $\mathbf{T}$ such that $\mathbf{T}^F= \mathfrak{H}_3$ if $q \equiv 1 \mod 3$ and $\mathbf{T}^F= \mathfrak{H}_6$ if $q \equiv -1 \mod 3$ in the notation of \cite{changree1974charg2}. Since we know that all characters are invariant under the actions of $\Gamma$ and $\mathcal{H}$, we can argue as in the proof of Proposition \ref{propextunip} to obtain $\mathcal{H}$-invariant character extensions.

Now let $q$ be even. The irreducible characters of $G$ are described in \cite{enomotoyamada1986charg2} and the irreducible $3'$-characters are denoted by $\theta_0, \theta_1, \theta_2, \theta_3,\theta_4, \theta_6, \theta_7, \theta_8$. Let $P, Q$ be subgroups of $G$ as defined there. The decompositions of some characters induced to $G$ are given in \cite[Table A]{enomotoyamada1986charg2} and we can read off
$$\Ind_P^G(1)=\theta_0+ \theta_1+ \theta_2 + \theta_3, \quad
\Ind_Q^G(1)=\theta_0+ \theta_1+ \theta_2 + \theta_4, \quad
\langle \Ind_P^G(\chi_1((q-1)/3)), \theta_8   \rangle=1$$
for some linear character $\chi_1((q-1)/3)$ of $P$ given in \cite[p.335]{enomotoyamada1986charg2}.
This already shows the claim for $\theta_0, \theta_1, \theta_2, \theta_3,$ and $\theta_4$. We also know that $\theta_6$ is semisimple and $ \theta_7$ is regular. Therefore they can be extended as described in Proposition \ref{propregularext} and \cite[Proposition 6.7]{ruhstorfer2017navarro}.

For the character $\chi_1((q-1)/3)$ of $P$ we have $\chi_1((q-1)/3)^{\psi} = \chi_1(2(q-1)/3)$ and we know from \cite[p.339]{enomotoyamada1986charg2} that both characters induce the same character of $G$. Thus, we find a $g \in G$ such that $\chi_1((q-1)/3)^{\psi g} = \chi_1((q-1)/3)$. Since it is linear, we find an $\mathcal{H}$-invariant extension of $\chi_1((q-1)/3)$ to $G \rtimes \langle \psi g \rangle$. It follows that $\theta_8$ also has an $\mathcal{H}$-invariant extension to $G \rtimes \langle \psi g \rangle$ that can be mapped to an $\mathcal{H}$-invariant character of $G \rtimes \langle \psi \rangle$ by conjugation with $g$. This shows the claim.
\end{proof}
The previous considerations directly give us the following result.
\begin{theorem}
The groups $\mathsf{G}_2(q)$ with $q \neq 2$, $q \equiv 2,4,5,7 \mod 9$ satisfy the inductive McKay--Navarro condition for $p=3$. 
\end{theorem}

\subsection{Some explicit groups} \label{sectionq2=2}
We now study the groups that we excluded before. The group ${}^2\mathsf{B}_2(2)$ is the Frobenius group of order $20$, hence solvable, and we do not have to consider it here. Again, we have to verify the original inductive McKay--Navarro condition from \cite[Definition 3.5]{navarro2019reduction}.
\begin{prop} \label{propsmall}
The groups ${}^2\mathsf{B}_2(8), {}^2\mathsf{G}_2(3)' , \mathsf{G}_2(2),$ and ${}^2\mathsf{F}_4(2)'$ satisfy the inductive McKay--Navarro condition for every prime $p$. The groups $\PSL_3(4)$ and $\mathsf{G}_2(4)$ satisfy the inductive McKay--Navarro condition for $p=3$.
\end{prop}
\begin{proof}
The computations were all made with \cite{GAP4}.
First, we consider $S={}^2\mathsf{B}_2(8)$. The universal covering group of $S$ is $G=2^2.{}^2\mathsf{B}_2(8)$. For $p=5$, let $N$ be the normalizer of a Sylow $5$-subgroup of $G$. Then, $\mathcal{H}$ acts trivially on all characters of $N$ and since we know by \cite[Theorem 4.1]{malle2008notlietype} that the inductive McKay-condition holds, it remains to consider the character extensions to their stabilizer in $G \rtimes  \Gamma$ or $N \rtimes  \Gamma$. Here, we see by explicit constructions that there is an $\mathcal{H}$-invariant extension for every $p'$-character of $N$ and $G$. This is sufficient since the outer automorphism group of $G$ is cyclic.
The remaining primes $p=7$ and $p=13$ can be treated in the same way. 

Similarly, the actions on the irreducible $p'$-characters and the extended characters can be computed straightforwardly for ${}^2\mathsf{G}_2(3)' \cong \PSL_2(8)$ and $\mathsf{G}_2(2)$ for all primes.

The Tits group ${}^2\mathsf{F}_4(2)'$ has trivial Schur multiplier and again we can explicitly determine the actions of $\mathcal{H}$ and the group automorphisms on the global and local characters. The extension ${}^2\mathsf{F}_4(2)'.2 \cong {}^2\mathsf{F}_4(2)$ contained in the ATLAS is not split but we can find an outer automorphism $\tau \in \Aut(G)$ of order $4$. Since ${}^2\mathsf{F}_4(2)' \rtimes \langle \tau \rangle$ is isoclinic to $ {}^2\mathsf{F}_4(2) \times C_2$, we can obtain its character table from the ATLAS \cite{ATLAS}. As before, the extensions of the local characters can be directly computed. We see that all characters have $(\Gamma \times \mathcal{H})$-invariant extensions to their stabilizers in the automorphism group. 

For $S=\PSL_3(4)$ and $p=3$, it suffices to consider the $3'$-part of the Schur multiplier. Thus, let $G=4^2.\PSL_3(4)$ and let $N$ be the normalizer of a Sylow $3$-group. As before, we can explicitly construct the extensions of the irreducible $3'$-characters of $G$ and $N$ to their inertia groups in $G \rtimes  \Gamma$ or $N \rtimes  \Gamma$. We see that the action of $(\Gamma \times \mathcal{H})_\psi$ on these extensions is trivial. 

For $\mathsf{G}_2(4)$, the character table of its universal covering group $G=2.\mathsf{G}_2(4)$ and its split extension $2.\mathsf{G}_2(4).2$ with its outer automorphism group of order $2$ are included in the ATLAS \cite{ATLAS}. Again, we let $N$ be the normalizer of a Sylow $3$-group and see that we find a $\Gamma \times \mathcal{H}$-equivariant bijection between $\Irr_{3'}(G)$ and $\Irr_{3'}(N)$. Looking at the character table of $2.\mathsf{G}_2(4).2$ and the corresponding extension of $N$, we see that the extensions of the $\Gamma$-invariant characters in $\Irr_{3'}(G)$ and $\Irr_{3'}(N)$ are not all $\mathcal{H}$-invariant. However, the actions of $\mathcal{H}$ on the extended characters are permutation isomorphic. Thus, the inductive McKay--Navarro condition holds for $p=3$.
\end{proof}

\subsection*{Acknowledgement} 
I would like to thank my PhD advisor Gunter Malle for his suggestions and comments. Furthermore, I am thankful to Britta Späth for pointing out a gap in the proof of Proposition \ref{propextunip} in an earlier version and Lucas Ruhstorfer for his helpful comments on it. This work was financially supported by the SFB-TRR 195 of the German Research Foundation (DFG).

\bibliography{biblio}
\bibliographystyle{amsalpha} 
\end{document}